\documentclass[final,1p,times,nopreprintline]{elsarticle}
\usepackage{amsmath,amssymb,amsthm,mathtools}
\usepackage{booktabs}
\usepackage{microtype}
\usepackage[colorlinks,linkcolor=blue,citecolor=blue,urlcolor=blue]{hyperref}

\theoremstyle{plain}
\newtheorem{theorem}{Theorem}[section]
\newtheorem{proposition}[theorem]{Proposition}
\newtheorem{lemma}[theorem]{Lemma}
\newtheorem{corollary}[theorem]{Corollary}
\theoremstyle{definition}
\newtheorem{definition}[theorem]{Definition}
\newtheorem{remark}[theorem]{Remark}
\newtheorem{example}[theorem]{Example}

\newcommand{\R}{\mathbb{R}}
\newcommand{\E}{\mathbb{E}}
\newcommand{\PP}{\mathbb{P}}
\newcommand{\Var}{\operatorname{Var}}
\newcommand{\MR}{M_{R}}
\newcommand{\cstar}{c_*}

\begin{document}

\begin{frontmatter}

\title{Feasible Frontiers for Sub-Gamma Envelopes: the Variance--Pole Trade-off for
Infinitely Divisible Laws}

\author[zstu]{Yichuan Chen\corref{cor1}\fnref{equal}}
\ead{smallboychen@gmail.com}

\author[zstu]{Xin Wang\fnref{equal}}
\ead{wxcaup@zstu.edu.cn}
\cortext[cor1]{Corresponding author.}
\fntext[equal]{These authors contributed equally to this work.}

\affiliation[zstu]{organization={School of Art and Design, Zhejiang Sci-Tech University},
            city={Hangzhou},
            postcode={310018},
            country={China}}

\begin{abstract}
A right sub-gamma bound is described by a quadratic proxy $v$ and a pole $c$, and the pair
is not unique: enlarging either preserves it. Fixing $v$ at the variance $V$ removes the
ambiguity at quadratic order but is a convention, not a consequence. For centered
infinitely divisible laws with finite nonzero variance we determine the entire boundary of
the feasible set, the map $v\mapsto\cstar(v)$. A Beta$(1,2)$ multiplier applied to the
normalised Kolmogorov canonical measure turns feasibility into a one-dimensional
comparison and yields an exact variational formula: the frontier is convex, nonincreasing,
and its feasible set is convex. The abscissa of convergence of the moment generating
function imposes a $v$-insensitive floor on $\cstar$, so a law whose variance-exact pole
sits on it has a flat frontier, whereas the third-cumulant obstruction exists only at
$v=V$. Whenever that pole lies strictly above both floors the frontier drops strictly as
soon as $v>V$; if the control is also purely local with
$\kappa_3^2/(9V^2)>\kappa_4/(12V)$, the drop has a square-root profile, and in the
remaining boundary case a cube-root profile, with derivative $-\infty$ at $V$ either way. We then characterise when the optimised frontier reproduces the exact Chernoff
deviation: equality holds at a level exactly when some pair on the frontier is tangent to
the cumulant generating function at a Legendre maximiser for that level. At the levels tabulated in a worked example $v=V$ costs four to thirty percent; for a
centered exponential a gap persists at every level, reaching thirteen percent.
\end{abstract}


\end{frontmatter}

\section{Introduction}

A centered random variable $X$ is right sub-gamma with variance factor $v>0$ and pole
$c\ge0$ if
\begin{equation}\label{eq:sg}
  \log\E e^{tX} \;\le\; \frac{v t^2}{2(1-ct)},
  \qquad 0\le t<1/c,
\end{equation}
with $1/0=\infty$. The bound is the standard route to Bernstein-type tail estimates
\cite{BLM2013} and the pair $(v,c)$ is the usual summary of a light-tailed law. It is
however not intrinsic. Increasing either coordinate preserves \eqref{eq:sg}, so a
convenient pair may be far from sharp, and a statement of the form ``$X$ is sub-gamma with
parameters $(v,c)$'' carries no information until one coordinate is pinned down.

The natural normalisation is $v=V:=\Var(X)$. It is forced at quadratic order, since
$\log\E e^{tX}=Vt^2/2+o(t^2)$ near the origin, and it leaves a single number to be
determined: the smallest compatible pole. This variance-exact problem has been solved for
several classes of laws \cite{Chen2026a,ChenWang2026a,ChenWang2026b}, where the answer is
governed by an ageing-type comparison in the sense of Klar and M\"uller
\cite{KlarMuller2003}, in parallel with
optimal variance-proxy problems for other envelopes
\cite{MarchalArbel2017,Arbel2020,Barreto2026,LV2026}
and with variance-matched scale optimisation for particular scalar families
\cite{Skorski2023,AMA2025}. The quantity $\kappa_3/(3V)$ that appears below as the local
obstruction is the same one that \cite{Skorski2023} identifies, for beta laws, as the
optimal Bernstein scale: there the scale is reported as $c=\kappa_3/V$ and enters the tail
bound as $c\epsilon/3$, which is our pole.

Pinning $v$ at $V$ is nevertheless a convention. The local expansion forces $v\ge V$; it
does not say that $v=V$ is the best place to stand. The purpose of this paper is to
describe the whole trade-off. Writing
\[
  \mathcal F=\{(v,c): \eqref{eq:sg}\ \text{holds}\},
  \qquad
  \cstar(v)=\inf\{c\ge0:(v,c)\in\mathcal F\},
\]
we determine the function $\cstar$ on $[V,\infty)$ for centered infinitely divisible laws
with finite nonzero variance, and we show that the frontier, rather than the single point
$\cstar(V)$, is the object that governs concentration.

\subsection*{Method}

We use the finite-variance canonical representation of Kolmogorov
\cite{Kolmogorov1932,GnedenkoKolmogorov1968}; see also
\cite[Remark 8.4 and Corollary 25.8]{Sato1999} and \cite{SvH2004}. If $a$ is the
Gaussian variance and $\nu$ the L\'evy measure, then
\begin{equation}\label{eq:LKintro}
  K_X(t)=\log\E e^{tX}=\frac{at^2}{2}+\int_{\R\setminus\{0\}}\bigl(e^{tx}-1-tx\bigr)\,\nu(dx),
  \qquad
  V=a+\int_{\R\setminus\{0\}}x^2\,\nu(dx).
\end{equation}
Normalising the Kolmogorov canonical measure $a\delta_0+x^2\nu(dx)$ by $V$, multiplying by
an independent $\mathrm{Beta}(1,2)$ variable, and using the integral form of the
second-order Taylor remainder gives a variable $R$ with
\begin{equation}\label{eq:factintro}
  K_X(t)=\frac{Vt^2}{2}\,\MR(t),\qquad \MR(t)=\E e^{tR}.
\end{equation}
Neither \eqref{eq:LKintro} nor \eqref{eq:factintro} is new. Identity \eqref{eq:factintro}
combines Kolmogorov's form $K_X''(t)=\int e^{tx}H_X(dx)$, in which $H_X$ is finite of mass
$V$, with the integral form of the second-order Taylor remainder,
$e^z-1-z=\tfrac{z^2}{2}\E e^{zB}$ for $B\sim\mathrm{Beta}(1,2)$; it is recorded only to fix
notation. What is
new is that \eqref{eq:factintro} converts the two-parameter feasibility problem into a
one-parameter comparison with an exponential law of adjustable mass, which is what makes
the entire frontier computable.

\subsection*{Results}

Write $\theta=v/V$ and $\tau=\sup\{t>0:K_X(t)<\infty\}$.

\begin{enumerate}
\item[(i)] \emph{Exact frontier} (Theorem~\ref{thm:var}):
\[
  \cstar(v)=\max\Bigl\{0,\ \sup_{t>0}\frac1t\Bigl(1-\frac{\theta}{\MR(t)}\Bigr)\Bigr\},
\]
and, when the right-hand side is finite, the infimum is attained. The frontier is convex
and nonincreasing and $\mathcal F$ is convex (Theorem~\ref{thm:convex} with
Proposition~\ref{prop:degenerate}(i), which empties the fibres below $V$); when $\tau>0$, $\cstar$ is continuous at the endpoint $v=V$
(Proposition~\ref{prop:cont}); and on $[V,\infty)$ finiteness of $\cstar(v)$ is equivalent
to $\tau>0$ irrespective of $v$ (Proposition~\ref{prop:degenerate}).

\item[(ii)] \emph{The two endpoint obstructions behave in opposite ways.} The boundary
floor $\cstar(v)\ge1/\tau$ holds for every $v$, so a law whose variance-exact pole is
boundary controlled has a constant frontier: relaxing the proxy buys nothing
(Proposition~\ref{prop:floor}). The local obstruction $\kappa_3/(3V)$, a lower bound for
$\cstar(V)$ that is attained exactly under local control, is by contrast a knife edge: for
every $\theta>1$ the objective tends to $-\infty$ at the origin
(Proposition~\ref{prop:knife}). More generally, no control mechanism is needed for the
frontier to move: whenever the variance-exact pole lies strictly above both floors,
$\cstar(V)>\max\{0,1/\tau\}$, it drops strictly as soon as $v>V$
(Theorem~\ref{thm:drop}). The two obstructions are not mutually exclusive: under pure local
control, where $\cstar(V)=m_1$, the coincidence $m_1=1/\tau$ makes the frontier constant
(Example~\ref{ex:knifefloor}). Without that hypothesis $m_1=1/\tau$ implies nothing, since
$m_1$ does not determine $\cstar(V)$.

\item[(iii)] \emph{Square-root edge law} (Theorem~\ref{thm:edge}). Under pure local
control with $m_1>0$ strictly dominating the floor $1/\tau$, and with
$\kappa:=\kappa_3^2/(9V^2)-\kappa_4/(12V)>0$,
\[
  \cstar(v)=\frac{\kappa_3}{3V}-2\sqrt{\kappa\,(\theta-1)}+o\bigl(\sqrt{\theta-1}\bigr),
  \qquad \theta\downarrow1,
\]
with maximiser $\hat t_\theta\sim\sqrt{(\theta-1)/\kappa}$. In particular
$\cstar'(V^+)=-\infty$. The condition $3V\kappa_4>4\kappa_3^2$ of \cite{ChenWang2026a} is
exactly $\kappa<0$, which rules local control out, so under pure local control the only
remaining case is $\kappa=0$; there $\Psi_1(t)=m_1-\gamma t^2+o(t^2)$ and, when
$\gamma>0$, the profile is a cube root,
\[
  \cstar(v)=\frac{\kappa_3}{3V}-3\Bigl(\frac{\gamma(\theta-1)^2}{4}\Bigr)^{1/3}
  +o\bigl((\theta-1)^{2/3}\bigr),
  \qquad
  \gamma=\frac{\kappa_3\kappa_4}{18V^2}-\frac{\kappa_3^3}{27V^3}-\frac{\kappa_5}{60V},
\]
with maximiser $\hat t_\theta\sim\{(\theta-1)/(2\gamma)\}^{1/3}$
(Theorem~\ref{thm:cube}). Both are instances of one $\varepsilon^{j/(j+1)}$ law
(Remark~\ref{rem:general-edge}), and in both the one-sided derivative at $V$ is $-\infty$.
This case is not vacuous: the Skellam law at its variance-exact phase transition sits
exactly on it.

\item[(iv)] \emph{Decay} (Section~\ref{sec:asym}). With $t_\theta=\inf\{t>0:\MR(t)>\theta\}$
one has $\cstar(v)\le1/t_\theta$ always, and $\cstar(v)\sim1/t_\theta$ under a mild
regularity condition (Theorem~\ref{thm:asym}); for bounded positive jumps of radius $b$,
where $\tau=\infty$ is automatic, this
gives $\cstar(v)\sim b/\log(v/V)$ (Corollary~\ref{cor:bounded}). The convergence is slow and the implicit form should be
preferred in numerical work.

\item[(v)] \emph{No loss over the frontier} (Theorem~\ref{thm:noloss}). Let
$I(y)=\sup_t\{ty-K_X(t)\}$. Then
\[
  \inf_{v\ge V}\Bigl\{\sqrt{2vx}+\cstar(v)\,x\Bigr\}\;\ge\;I^{-1}(x),
\]
with equality at a level $x$ \emph{if and only if} some pair on the frontier is tangent to
$K_X$ at the Legendre maximiser for that level; for a positive pole this is the same as the
variational supremum being attained at that point. Where
equality holds, the envelope family loses nothing and the loss usually attributed to
Bernstein-type bounds is a loss of the single point $v=V$: in the example of
Section~\ref{sec:noloss} that point costs between four and thirty percent at the levels tabulated there, and under
the
hypotheses of (iii) it is strictly suboptimal at every positive deviation level. Where
equality fails the gap is real; for a centered exponential variable it reaches thirteen
percent (Example~\ref{ex:gaploss}). Which case occurs is not settled by the control
mechanism: within the tempered stable family both a boundary-controlled law with no loss
at any level and a boundary-controlled law with a genuine gap occur
(Example~\ref{ex:noloss-exact} and Example~\ref{ex:gaploss}).
\end{enumerate}

Section~\ref{sec:examples} treats Gamma and bilateral Gamma laws, one-sided tempered
stable and CGMY laws, compound Poisson laws with Gamma jumps, the centered Skellam family,
and a bounded two-point jump law. The frontier is obtained in closed form where it is
constant --- Gamma and bilateral Gamma, and one-sided tempered stable with $Y\ge-1$ --- and
numerically otherwise; for a bilateral CGMY law only the variance-exact pole is computed,
as a counterexample to the one-sided formula. Two by-products
deserve mention: the tempered stable variance-exact pole is
$\cstar(V)=M^{-1}\max\{1,(2-Y)/3\}$, which contains the compound Poisson Gamma-jump pole
of \cite{Chen2026a} and the Gamma scale as the cases $Y=-\alpha$ and $Y=0$; and the
Skellam phase transition at $p_\pm=(2\pm\sqrt3)/4$ found in \cite{ChenWang2026a} exists
only at $v=V$, since by (ii) local control cannot survive any relaxation of the proxy.

Deviation inequalities for infinitely divisible laws expressed directly through the L\'evy
measure go back to \cite{Houdre2002}, with related results for norms of infinitely
divisible vectors \cite{HMR2008} and for compound Poisson laws
\cite{KontoyiannisMadiman2006}; interpolation and covariance representations for
infinitely divisible variables go back to \cite{HPS1998}, and Stein's method gives a
further route \cite{ArrasHoudre2019,Barman2025}. Those methods control deviations; they do not
describe the feasible set of envelope parameters, which is the object here.

\section{Setting}\label{sec:setting}

Throughout, $X$ is a centered, nondegenerate, infinitely divisible random variable with
finite variance $0<V=\Var(X)<\infty$, Gaussian variance $a\ge0$ and L\'evy measure $\nu$,
so that \eqref{eq:LKintro} holds. The integrand $e^{tx}-1-tx$ is nonnegative, so
$K_X(t)\in[0,\infty]$ is defined for every real $t$; where finite it is the cumulant
generating function. Set
\begin{equation}\label{eq:tau}
  \tau=\sup\{t>0:K_X(t)<\infty\}\in[0,\infty].
\end{equation}

Let $H_X(dx)=a\delta_0(dx)+x^2\nu(dx)$, a finite measure of total mass $V$. Let
$Z\sim H_X/V$, let $B\sim\mathrm{Beta}(1,2)$ be independent of $Z$, and put
\begin{equation}\label{eq:R}
  R=BZ,\qquad \MR(t)=\E e^{tR}\in(0,\infty].
\end{equation}
The density of $B$ is $2(1-b)$ on $(0,1)$.

\begin{lemma}[Canonical factorization]\label{lem:fact}
For every $t\in\R$, with extended values allowed and both sides read by continuity at
$t=0$,
\begin{equation}\label{eq:fact}
  K_X(t)=\frac{Vt^2}{2}\,\MR(t).
\end{equation}
In particular $\MR(t)<\infty$ if and only if $K_X(t)<\infty$, so $\tau$ is also the
positive abscissa of convergence of $\MR$.
\end{lemma}

\begin{proof}
Direct integration gives $\E e^{zB}=2(e^z-1-z)/z^2$ for $z\ne0$, with value $1$ at $z=0$;
equivalently $e^{z}-1-z=z^2\E e^{zB}/2$, the integral form of the second-order Taylor
remainder. Conditioning on $Z$ and integrating against $H_X/V$,
\[
  \frac{Vt^2}{2}\E e^{tBZ}
  =\frac{t^2}{2}\int_\R \E e^{tBy}\,H_X(dy)
  =\frac{at^2}{2}+\int_{\R\setminus\{0\}}\bigl(e^{tx}-1-tx\bigr)\nu(dx),
\]
which is \eqref{eq:fact}. All integrands are nonnegative, so the interchange is legitimate
and the equality of finiteness follows.
\end{proof}

We use the conventions $1/0=\infty$, $1/\infty=0$ and $\theta/\infty=0$ for finite
$\theta$, and we write throughout
\begin{equation}\label{eq:psi}
  \theta=\frac{v}{V},\qquad
  \Psi_\theta(t)=\frac1t\Bigl(1-\frac{\theta}{\MR(t)}\Bigr),\quad t>0 .
\end{equation}
Thus $\Psi_\theta(t)=1/t$ whenever $\MR(t)=\infty$, and $\Psi_\theta(t)\le1/t$ always.

One structural fact about this family is used repeatedly below and is recorded here.

\begin{lemma}[Monotone penalty]\label{lem:monpen}
The map $t\mapsto t\MR(t)$ is nondecreasing on $(0,\infty)$, so $\phi(t):=\{t\MR(t)\}^{-1}$
is nonincreasing, with $\phi\equiv0$ on $(\tau,\infty)$ when $\tau<\infty$; at
$t=\tau$ itself $\phi$ may be positive, which is consistent with monotonicity. Consequently
\begin{equation}\label{eq:pendecomp}
  \Psi_\theta=\Psi_1-(\theta-1)\phi ,\qquad \theta\ge1 .
\end{equation}
\end{lemma}

\begin{proof}
By Lemma~\ref{lem:fact}, $t\MR(t)=2K_X(t)/(Vt)$. The function $K_X$ is convex with
$K_X(0)=0$, so for $0<s<t$ we have $K_X(s)\le(s/t)K_X(t)$, that is $K_X(s)/s\le K_X(t)/t$,
with extended values allowed. Identity \eqref{eq:pendecomp} is immediate from
\eqref{eq:psi}.
\end{proof}

\section{The feasible region and the exact frontier}\label{sec:frontier}

\begin{definition}\label{def:feasible}
A pair $(v,c)\in(0,\infty)\times[0,\infty)$ is \emph{feasible} for $X$ if
\begin{equation}\label{eq:env}
  K_X(t)\le\frac{vt^2}{2(1-ct)}\qquad\text{for every }0\le t<1/c .
\end{equation}
Let $\mathcal F=\mathcal F(X)$ be the set of feasible pairs and, for $v>0$,
\begin{equation}\label{eq:cstardef}
  \cstar(v)=\inf\{c\ge0:(v,c)\in\mathcal F\}\in[0,\infty],
\end{equation}
with $\inf\emptyset=\infty$. For $c=0$ condition \eqref{eq:env} is required for all
$t\ge0$.
\end{definition}

\begin{lemma}[Reduction]\label{lem:reduce}
For $v>0$ and $c\ge0$, the pair $(v,c)$ is feasible if and only if
\begin{equation}\label{eq:reduced}
  \MR(t)\le\frac{\theta}{1-ct}\qquad\text{for every }0<t<1/c .
\end{equation}
\end{lemma}

\begin{proof}
Both sides of \eqref{eq:env} vanish at $t=0$. For $0<t<1/c$ divide by $Vt^2/2>0$ and use
\eqref{eq:fact}; the resulting inequalities are identical, including the cases where
either side is infinite.
\end{proof}

The right-hand side of \eqref{eq:reduced} is $\theta$ times the moment generating function
of an exponential variable with mean $c$. Feasibility is therefore a comparison of $\MR$
with a scaled exponential transform, and the frontier is the exact solution of that
comparison.

\begin{theorem}[Exact frontier]\label{thm:var}
For every $v>0$,
\begin{equation}\label{eq:varformula}
  \cstar(v)=\max\Bigl\{0,\ \sup_{t>0}\Psi_\theta(t)\Bigr\}.
\end{equation}
If the right-hand side is finite it is itself feasible; the infimum in
\eqref{eq:cstardef} is then attained and is a minimum.
\end{theorem}

\begin{proof}
Let $c$ be feasible and $t>0$. If $t\ge1/c$, which is vacuous for $c=0$, then
$\Psi_\theta(t)\le1/t\le c$. If $0<t<1/c$ then $\MR(t)<\infty$, since otherwise the left
side of \eqref{eq:reduced} is infinite and the right side finite; rearranging
\eqref{eq:reduced} gives $1-ct\le\theta/\MR(t)$, that is $\Psi_\theta(t)\le c$. Hence
$\sup_{t>0}\Psi_\theta\le c$ for every feasible $c$, and as feasible poles are nonnegative
the right side of \eqref{eq:varformula} is a lower bound for $\cstar(v)$.

Conversely let $C=\max\{0,\sup_{t>0}\Psi_\theta\}<\infty$ and fix $0<t<1/C$. If
$\MR(t)=\infty$ then $\Psi_\theta(t)=1/t>C$, a contradiction, so $\MR(t)<\infty$. From
$\Psi_\theta(t)\le C$ we get $\theta/\MR(t)\ge1-Ct>0$, hence $\MR(t)\le\theta/(1-Ct)$. By
Lemma~\ref{lem:reduce} the pair $(v,C)$ is feasible, which gives the reverse inequality
and attainment.
\end{proof}

\begin{proposition}[Degenerate cases]\label{prop:degenerate}
\begin{enumerate}
\item[(i)] If $v<V$ then $\cstar(v)=\infty$.
\item[(ii)] $\cstar(v)<\infty$ for some $v>0$ if and only if $\tau>0$, and then
$\cstar(v)<\infty$ for every $v\ge V$. On $[V,\infty)$ finiteness of the frontier does not
depend on $v$.
\end{enumerate}
\end{proposition}

\begin{proof}
(i) Let $\theta<1$. If $\tau=0$ then $\MR(t)=\infty$ for every $t>0$, so
$\Psi_\theta(t)=1/t$ and the supremum is infinite. If $\tau>0$ then $\MR$ is finite on
$(0,\tau)$ and $\MR(t)\to1$ as $t\downarrow0$ by dominated convergence, dominating by
$1+e^{t_0R}$ for a fixed $t_0\in(0,\tau)$; hence
$\Psi_\theta(t)=t^{-1}(1-\theta+o(1))\to+\infty$.

(ii) If $\tau=0$ then $\MR(t)=\infty$ for every $t>0$, so $\Psi_\theta(t)=1/t$ for every
$\theta$, whose
supremum is infinite. If $\tau>0$ and $\theta\ge1$ fix $0<t_0<\tau$. On $(0,t_0]$ either $\MR(t)<1$, in which
case $\Psi_\theta\le\Psi_1<0$, or $\MR(t)\ge1$ and
$\Psi_\theta(t)\le\Psi_1(t)\le\{\MR(t)-1\}/t\le\{\MR(t_0)-1\}/t_0$, the last step because
$\MR$ is convex with $\MR(0)=1$, so that $t\mapsto\{\MR(t)-1\}/t$ is nondecreasing; on
$[t_0,\infty)$ we have $\Psi_\theta(t)\le1/t\le1/t_0$. (For $\theta<1$ the supremum is
infinite by (i).)
\end{proof}

From now on we consider $v\ge V$.

\begin{theorem}[Structure of the frontier]\label{thm:convex}
The map $v\mapsto\cstar(v)$ is nonincreasing and convex on $[V,\infty)$, and
\begin{equation}\label{eq:epi}
  \mathcal F\cap\bigl([V,\infty)\times[0,\infty)\bigr)=\{(v,c):v\ge V,\ c\ge\cstar(v)\},
\end{equation}
so this part of the feasible region is convex.
\end{theorem}

\begin{proof}
Fix $t>0$. By \eqref{eq:psi} the map
$v\mapsto\Psi_{v/V}(t)=t^{-1}-v\bigl(tV\MR(t)\bigr)^{-1}$ is affine in $v$ with slope
$-\bigl(tV\MR(t)\bigr)^{-1}\le0$; if $\MR(t)=\infty$ it is the constant $1/t$, again affine
with slope $0$. A pointwise supremum of affine nonincreasing functions is convex and
nonincreasing, and taking the maximum with the constant $0$ preserves both properties;
Theorem~\ref{thm:var} then gives the first claim.

Feasibility is monotone in $c$: if $c\le c'$ and $(v,c)\in\mathcal F$ then for
$0\le t<1/c'\le1/c$,
\[
  K_X(t)\le\frac{vt^2}{2(1-ct)}\le\frac{vt^2}{2(1-c't)} .
\]
Together with the attainment in Theorem~\ref{thm:var} this gives \eqref{eq:epi} when
$\cstar(v)<\infty$; when $\cstar(v)=\infty$ both sides are empty in that fibre by
Proposition~\ref{prop:degenerate}(ii). The right side of \eqref{eq:epi} is the epigraph of
a convex function on a half-line.
\end{proof}

\begin{proposition}[Continuity at the variance-exact endpoint]\label{prop:cont}
Assume $\tau>0$. Then $\cstar$ is continuous on $[V,\infty)$; in particular
$\lim_{v\downarrow V}\cstar(v)=\cstar(V)$.
\end{proposition}

\begin{proof}
A finite convex function is continuous on the interior of its domain, so only the endpoint
needs an argument. As $\cstar$ is nonincreasing the limit
$L=\lim_{\theta\downarrow1}\cstar(\theta V)$ exists and $L\le\cstar(V)$. For the reverse
inequality fix $t_0>0$. Since $\theta\mapsto\Psi_\theta(t_0)$ is affine, hence continuous,
\[
  L\;\ge\;\lim_{\theta\downarrow1}\max\{0,\Psi_\theta(t_0)\}\;=\;\max\{0,\Psi_1(t_0)\}.
\]
Taking the supremum over $t_0>0$ gives $L\ge\max\{0,\sup_{t>0}\Psi_1\}=\cstar(V)$.
\end{proof}

\begin{remark}
Continuity at the endpoint does not follow from convexity, which controls only the
interior, and the objectives on the two sides of $v=V$ are of different shapes: the
supremum defining $\cstar(V)$ may be approached only as $t\downarrow0$, whereas for every
$\theta>1$ the objective tends to $-\infty$ at the origin
(Proposition~\ref{prop:knife}). It is the location of the maximiser, not its value, that
jumps at $v=V$.
\end{remark}

\section{Mechanisms along the frontier}\label{sec:mech}

The variance-exact pole can be produced by a local mechanism at $t\downarrow0$, a boundary
mechanism on $(\tau,\infty)$, or an interior maximiser \cite{Chen2026a}. These are not
mutually exclusive --- the supremum may be approached or attained in more than one of the
three ways at once, and Example~\ref{ex:knifefloor} realises two of them together --- but
the first two respond to a change of $v$ in opposite ways. We fix the terminology, which
is used throughout the paper.

\begin{definition}[Control mechanisms]\label{def:control}
Assume $\tau>0$ and $\cstar(V)>0$, so that $\cstar(V)=\sup_{t>0}\Psi_1$ by
Theorem~\ref{thm:var}. The variance-exact pole is
\emph{locally controlled} if $\cstar(V)=\lim_{t\downarrow0}\Psi_1(t)$, which by
Proposition~\ref{prop:knife} below is the statement $\cstar(V)=m_1=\kappa_3/(3V)$ whenever
\eqref{eq:m1} applies;
\emph{boundary controlled} if $\cstar(V)=1/\tau$; and
\emph{interior controlled} if the supremum is attained at some $t\in(0,\tau)$.
\end{definition}

These three exhaust the possibilities. Indeed $\Psi_1$ is continuous on $(0,\tau]$ and
equals $1/t$ on $(\tau,\infty)$ by Lemma~\ref{lem:psicont} below, so a supremum that is neither
approached at the origin nor attained in $(0,\tau)$ must be $\sup_{t\ge\tau}\Psi_1$; that
supremum equals $1/\tau$, since $\Psi_1(\tau)<1/\tau$ when $\MR(\tau)<\infty$ and
$\Psi_1(\tau)=1/\tau$ otherwise. They are not exclusive: Example~\ref{ex:knifefloor} is
both locally and boundary controlled, and so is the tempered stable law of
Example~\ref{ex:ts} at $Y=-1$.

\subsection{The boundary floor}

\begin{proposition}[Floor and flatness]\label{prop:floor}
For every $v>0$ one has $\cstar(v)\ge1/\tau$. Consequently, if the variance-exact pole is
boundary controlled, that is $\cstar(V)=1/\tau$, then
$\cstar(v)=1/\tau$ for every $v\ge V$: the frontier is constant and relaxing the quadratic
proxy buys nothing.
\end{proposition}

\begin{proof}
For $\tau=\infty$ there is nothing to prove. If $\tau<\infty$ then $\MR(t)=\infty$ for
$t>\tau$, so $\Psi_\theta(t)=1/t$ there and
$\sup_{t>0}\Psi_\theta\ge\sup_{t>\tau}1/t=1/\tau$. The second claim follows from
monotonicity: $1/\tau\le\cstar(v)\le\cstar(V)=1/\tau$.
\end{proof}

\subsection{Local control is a knife edge}

Assume, here and wherever $m_1$ or $\kappa_3$ appears below, that
$\int_{|x|>1}|x|^3\nu(dx)<\infty$, so that $\kappa_3$ exists
and
\begin{equation}\label{eq:m1}
  m_1=\E R=\frac{\kappa_3}{3V}.
\end{equation}

We first record the regularity of $\Psi_\theta$, which is used repeatedly below.

\begin{lemma}[Continuity and the jump at the abscissa]\label{lem:psicont}
Let $\theta\ge1$. As an extended-real-valued function $\Psi_\theta$ is continuous on
$(0,\tau]$ and equals $1/t$ on $(\tau,\infty)$. It need \emph{not} be continuous at $\tau$
from the right: if $\tau<\infty$ and $\MR(\tau)<\infty$ then
$\Psi_\theta(\tau)<1/\tau=\lim_{t\downarrow\tau}\Psi_\theta(t)$, an upward jump.
\end{lemma}

\begin{proof}
$\MR$ is convex and finite on $(0,\tau)$, hence continuous there; and
$\MR(t)\to\MR(\tau)$ as $t\uparrow\tau$, by monotone convergence on $\{R\ge0\}$ and
dominated convergence on $\{R<0\}$, where the integrand is bounded by $1$. For $t>\tau$
we have $\MR(t)=\infty$, so $\Psi_\theta(t)=1/t$. If $\MR(\tau)<\infty$ then
$\Psi_\theta(\tau)=\tau^{-1}\{1-\theta/\MR(\tau)\}<1/\tau$.
\end{proof}

\begin{proposition}[Knife edge]\label{prop:knife}
Assume $\tau>0$. Then $\lim_{t\downarrow0}\Psi_1(t)=m_1$, while
$\lim_{t\downarrow0}\Psi_\theta(t)=-\infty$ for every $\theta>1$. Hence the local
obstruction $\cstar(V)\ge\max\{0,\kappa_3/(3V)\}$ has no analogue for $v>V$: there the
supremum in \eqref{eq:varformula} is attained at a positive finite argument, or forced by
the floor of Proposition~\ref{prop:floor}, or nonpositive, in which case $\cstar(v)=0$.
\end{proposition}

\begin{proof}
Since $\tau>0$, $\MR$ is finite near the origin, and dominated convergence --- with
$(e^{t_0R}-1)/t_0$ on $\{R\ge0\}$ for a fixed $t_0\in(0,\tau)$ and $|R|$ on $\{R<0\}$ --- gives
$\MR(t)=1+m_1t+o(t)$, so
$\theta/\MR(t)=\theta(1-m_1t+o(t))$ and
\[
  \Psi_\theta(t)=\frac{1-\theta}{t}+\theta m_1+o(1).
\]
For $\theta=1$ the first term vanishes; for $\theta>1$ it tends to $-\infty$.

For the last claim fix $\theta>1$ and suppose $\sup_{t>0}\Psi_\theta>0$. By
Lemma~\ref{lem:psicont}, $\Psi_\theta$ is continuous on $(0,\tau]$ and equals $1/t$ on
$(\tau,\infty)$, so $\sup_{t>0}\Psi_\theta=\max\{\sup_{(0,\tau]}\Psi_\theta,\,1/\tau\}$,
with $1/\tau=0$ when $\tau=\infty$. On $(0,\tau]$ the function $\Psi_\theta$ tends to
$-\infty$ at the origin by the first part and satisfies $\Psi_\theta\le1/t$, so it attains
its supremum there at some positive finite argument unless that supremum is at most
$1/\tau$, in which case the value $1/\tau$ of Proposition~\ref{prop:floor} is what forces
the pole. If instead $\sup_{t>0}\Psi_\theta\le0$ then $\cstar(v)=0$ by
Theorem~\ref{thm:var}.
\end{proof}

The strict drop needs no control mechanism at all: it needs only that the variance-exact
pole lies strictly above both floors, the trivial one built into \eqref{eq:cstardef} and
the boundary one of Proposition~\ref{prop:floor}.

\begin{theorem}[Strict drop above the floors]\label{thm:drop}
If $\cstar(V)>\max\{0,1/\tau\}$, then $\cstar(v)<\cstar(V)$ for every $v>V$.
\end{theorem}

\begin{proof}
Write $c_0=\cstar(V)$; by hypothesis $c_0>0$, so $c_0=\sup_{t>0}\Psi_1$. The hypothesis
also forces $\tau>0$, since $\tau=0$ would give $\max\{0,1/\tau\}=\infty$. Fix $\theta>1$.
Since $\phi\ge0$, \eqref{eq:pendecomp} gives $\Psi_\theta\le\Psi_1$ and hence
$\sup_{t>0}\Psi_\theta\le c_0$. Suppose the supremum equals $c_0$ and choose $t_k>0$ with
$\Psi_\theta(t_k)\to c_0$. As $\Psi_\theta=\Psi_1-(\theta-1)\phi$ with both
$\Psi_1\le c_0$ and $\phi\ge0$, we get $\Psi_1(t_k)\to c_0$ and $\phi(t_k)\to0$. Passing
to a subsequence, $t_k\to t^*\in[0,\infty]$, and we distinguish three cases.

If $t^*=0$ then $\MR(t_k)\to1$, so $\phi(t_k)=\{t_k\MR(t_k)\}^{-1}\to\infty$ and
$\Psi_\theta(t_k)\to-\infty$, contradicting $\Psi_\theta(t_k)\to c_0$.

If $t^*\in(0,\tau)$ then $\MR$ is continuous and finite at $t^*$, so
$\phi(t_k)\to\phi(t^*)>0$, contradicting $\phi(t_k)\to0$.

If $t^*\in[\tau,\infty]$ then, since $\Psi_1(t)\le1/t$ for every $t>0$,
$c_0=\lim_k\Psi_1(t_k)\le\lim_k 1/t_k=1/t^*\le\max\{0,1/\tau\}$, with $1/\infty=0$,
contradicting the hypothesis $c_0>\max\{0,1/\tau\}$.

Hence $\sup_{t>0}\Psi_\theta<c_0$, and as $c_0>0$,
$\cstar(\theta V)=\max\{0,\sup_{t>0}\Psi_\theta\}<c_0$.
\end{proof}

Under pure local control the hypothesis of Theorem~\ref{thm:drop} takes an explicit form,
and a sharper statement about $\Psi_1$ away from the origin becomes available; both are
needed in Section~\ref{sec:asym}.

\begin{lemma}[Pure local control]\label{lem:away}
Assume $\tau>0$, $m_1>0$,
\begin{equation}\label{eq:purelocal}
  \Psi_1(t)<m_1\quad\text{for every }t>0,
\end{equation}
and, when $\tau<\infty$, the strict inequality $m_1>1/\tau$. Then
$\cstar(V)=m_1>\max\{0,1/\tau\}$, Theorem~\ref{thm:drop} applies, and
\begin{equation}\label{eq:away}
  \sup_{t\ge\delta}\Psi_1(t)<m_1\qquad\text{for every }\delta>0 .
\end{equation}
\end{lemma}

\begin{proof}
That $\cstar(V)=m_1$ is Proposition~\ref{prop:knife} together with
\eqref{eq:purelocal}, and $m_1>\max\{0,1/\tau\}$ holds by hypothesis. For
\eqref{eq:away} use Lemma~\ref{lem:psicont}. If $\tau=\infty$, set $T=2/m_1$: for
$t\ge T$ we have $\Psi_1(t)\le1/t\le m_1/2$, while on the compact interval $[\delta,T]$
the continuous function $\Psi_1$ attains a maximum, which is strictly below $m_1$ by
\eqref{eq:purelocal}; for $\delta\ge T$ the first bound alone suffices. If $\tau<\infty$,
then $\sup_{t>\tau}\Psi_1=\sup_{t>\tau}1/t=1/\tau<m_1$ by hypothesis, while for
$\delta<\tau$ the maximum of $\Psi_1$ over the compact interval $[\delta,\tau]$ is
attained and is strictly below $m_1$; for $\delta\ge\tau$ the first bound alone suffices.
\end{proof}

\begin{remark}\label{rem:implied}
Hypothesis \eqref{eq:purelocal} implies $m_1\ge1/\tau$ when $\tau<\infty$: for $t>\tau$
one has $\Psi_1(t)=1/t<m_1$, and letting $t\downarrow\tau$ gives $1/\tau\le m_1$. Pure
local control therefore dominates the floor automatically. The strict inequality assumed in
Lemma~\ref{lem:away} is a genuine further requirement, and it is exactly the statement that
the law is not simultaneously boundary controlled: by Proposition~\ref{prop:floor} the
frontier is constant as soon as $\cstar(V)=1/\tau$, so $m_1=1/\tau$ forces
$\cstar\equiv1/\tau$ and the hypothesis of Theorem~\ref{thm:drop} fails, as does its
conclusion. The two mechanisms are therefore not mutually exclusive, and the boundary case
$m_1=1/\tau$ occurs; Example~\ref{ex:knifefloor} exhibits it.
\end{remark}

\begin{example}[Local and boundary control can coincide]\label{ex:knifefloor}
Let $X$ be centered compound Poisson with L\'evy measure $\nu(dx)=x^{-2}e^{-(x-2)}dx$ on
$(2,\infty)$. Then $Z=2+E$ with $E$ standard exponential, $V=1$, $\kappa_3=3$,
$\kappa_4=10$, so $m_1=1$, $m_2=5/3$ and $\kappa=1/6>0$. Since $\E B^n=2/\{(n+1)(n+2)\}$
and $\E Z^n=n!\sum_{j\le n}2^j/j!$,
\[
  \frac{\E R^n}{n!}=\frac{2}{(n+1)(n+2)}\sum_{j\le n}\frac{2^j}{j!}
  \;\le\;\frac{2e^2}{(n+1)(n+2)},
\]
which equals $1$ for $n\in\{0,1\}$, equals $5/6$ at $n=2$, and for $n\ge3$ is at most
$2e^2/20<5/6$. Hence
$\MR(t)<(1-t)^{-1}$ strictly on $(0,1)$, that is $\Psi_1(t)<1=m_1$ there; moreover
$\tau=1$ with
\[
  \MR(1)=2e^2\int_2^\infty\frac{1-e^{-y}-ye^{-y}}{y^2}\,dy
        =2e^2\Bigl(\frac12-\frac{e^{-2}}{2}\Bigr)=e^2-1<\infty,
\]
so $\Psi_1(1)=1-(e^2-1)^{-1}=0.8435<1$, and $\Psi_1(t)=1/t<1$ for $t>1$. Thus \eqref{eq:purelocal} holds, together with $\tau>0$, $m_1>0$ and $\kappa>0$.
But $m_1=1=1/\tau$, and the same moment bound makes $c=1$ feasible for every
$\theta\ge1$, so with Proposition~\ref{prop:floor}
\[
  \cstar(v)\equiv1\qquad\text{for every }v\ge V .
\]
The frontier is constant: without the strict inequality $m_1>1/\tau$ the conclusions of
Theorem~\ref{thm:drop}, Theorem~\ref{thm:edge} and Corollary~\ref{cor:neveropt} all
fail for this law, and its hypothesis $\cstar(V)>\max\{0,1/\tau\}$ is exactly what fails.
\end{example}

\subsection{The zero set and a first upper bound}

\begin{proposition}[Zero dichotomy]\label{prop:zero}
$\cstar(v)=0$ for some $v>0$ if and only if $\nu((0,\infty))=0$, and then $\cstar\equiv0$
on $[V,\infty)$. Equivalently, if $X$ has positive jumps then the frontier is strictly
positive at every $v$, however large.
\end{proposition}

\begin{proof}
If $\nu((0,\infty))=0$ then $Z\le0$, hence $R\le0$ almost surely, so $\MR(t)\le1\le\theta$
for $t>0$ and $\theta\ge1$; thus $\Psi_\theta\le0$. Conversely, if $\nu((0,\infty))>0$
then $\PP(R>0)>0$, and $\E[e^{tR}\mathbf 1_{\{R\ge0\}}]\uparrow\infty$ as $t\to\infty$ by
monotone convergence while $\E[e^{tR}\mathbf 1_{\{R<0\}}]\in[0,1]$; hence
$\MR(t)\to\infty$, possibly through the value $+\infty$, so $\MR(t)>\theta$ for some $t$
and $\Psi_\theta(t)>0$.
\end{proof}

\begin{proposition}[Threshold bound]\label{prop:tth}
For $\theta\ge1$ set $t_\theta=\inf\{t>0:\MR(t)>\theta\}\in[0,\infty]$, with
$1/0=\infty$. Then $\cstar(v)\le1/t_\theta$.
\end{proposition}

\begin{proof}
For $0<t<t_\theta$ we have $\MR(t)\le\theta$ and $\Psi_\theta(t)\le0$; for $t\ge t_\theta$
we have $\Psi_\theta(t)\le1/t\le1/t_\theta$. Apply Theorem~\ref{thm:var}.
\end{proof}

\begin{proposition}[Critical proxy]\label{prop:vstar}
Suppose $0<\tau<\infty$ and define $v^*=\inf\{v\ge V:\cstar(v)=1/\tau\}\in[V,\infty]$. Then
$v^*\le V\max\{1,\MR(\tau^-)\}$, where $\MR(\tau^-)=\lim_{t\uparrow\tau}\MR(t)$, and in all
cases
\[
  \cstar(v)\downarrow\frac1\tau\qquad\text{as }v\to\infty .
\]
Beyond $v^*$ the frontier is constant, so no further relaxation of the proxy is useful.
\end{proposition}

\begin{proof}
Let $\theta\ge\max\{1,\MR(\tau^-)\}$. Since $\MR$ is convex on $[0,\tau)$ with $\MR(0)=1$, it
attains its supremum there at an endpoint, so $\MR(t)\le\max\{1,\MR(\tau^-)\}\le\theta$ for
every $t<\tau$; $\MR$ need not be monotone, so it is convexity that is used here. Hence
$t_\theta=\tau$ and Proposition~\ref{prop:tth} gives $\cstar(v)\le1/\tau$; with
Proposition~\ref{prop:floor} this is an equality, proving
$v^*\le V\max\{1,\MR(\tau^-)\}$. In general $t_\theta\uparrow\tau$ as $\theta\to\infty$:
for any $s<\tau$ convexity again gives $\sup_{t\le s}\MR=\max\{1,\MR(s)\}<\infty$, so
$t_\theta\ge s$ as soon as $\theta$ exceeds that value. Hence
$\cstar(v)\le1/t_\theta\to1/\tau$, while $\cstar(v)\ge1/\tau$
throughout; monotonicity gives the decreasing convergence. Constancy beyond $v^*$ follows
from Proposition~\ref{prop:floor} and monotonicity.
\end{proof}

\section{Edge behaviour and decay}\label{sec:asym}

\subsection{The square-root edge law}

Assume, here and wherever $m_2$ or $\kappa$ appears below --- but not in
Section~\ref{sec:asym}'s second subsection, which needs no moment condition beyond the
standing one --- that $\int_{|x|>1}x^4\,\nu(dx)<\infty$, so that
\begin{equation}\label{eq:m2}
  m_2=\E R^2=\frac{\kappa_4}{6V},
\end{equation}
and set
\begin{equation}\label{eq:kappa}
  \kappa=m_1^2-\frac{m_2}{2}=\frac{\kappa_3^2}{9V^2}-\frac{\kappa_4}{12V}.
\end{equation}
When $\tau>0$, expanding \eqref{eq:psi} at the origin gives $\Psi_1(t)=m_1-\kappa t+o(t)$
--- the expansion of $\MR$ to second order follows from dominated convergence, using
$R^2(e^{t_0R}-1-t_0R)/(t_0R)^2$ on $\{R\ge0\}$ for a fixed $t_0\in(0,\tau)$ and
$R^2/2$ on $\{R<0\}$, the latter because $z\mapsto(e^z-1-z)/z^2$ increases on $\R$ --- so $\kappa>0$
implies that $\Psi_1$ decreases at the origin; the converse can fail at $\kappa=0$, as
Example~\ref{ex:skellam} at $p=p_+$ shows. Equivalently $\kappa<0$ is
the condition $3V\kappa_4>4\kappa_3^2$ of \cite{ChenWang2026a}. It excludes local control
but does not by itself force an interior maximiser: for the tempered stable family of
Example~\ref{ex:ts} one computes $\kappa=(2-Y)(-1-Y)/(36M^2)$, so $\kappa<0$ exactly when
$Y>-1$, and Proposition~\ref{prop:ts} shows that the whole range $-1<Y<2$ is boundary
controlled with a constant frontier. Pure local
control \eqref{eq:purelocal} already forces $\kappa\ge0$: if $\kappa<0$ then
$\Psi_1(t)>m_1$ for all small $t>0$, contradicting \eqref{eq:purelocal}. Assuming
$\kappa>0$ therefore excludes only the boundary case $\kappa=0$, which is treated
separately in Theorem~\ref{thm:cube} and which does occur: it is where the Skellam
transition of Example~\ref{ex:skellam} sits.

For that boundary case one order more is needed. Assume, wherever $m_3$ or $\gamma$
appears below, that $\int_{|x|>1}|x|^5\nu(dx)<\infty$, so that
\begin{equation}\label{eq:m3}
  m_3=\E R^3=\frac{\kappa_5}{10V},
\end{equation}
and set
\begin{equation}\label{eq:gamma}
  \gamma=m_1m_2-m_1^3-\frac{m_3}{6}
        =\frac{\kappa_3\kappa_4}{18V^2}-\frac{\kappa_3^3}{27V^3}-\frac{\kappa_5}{60V}.
\end{equation}
Then, when $\tau>0$, the expansion above continues to one further order,
\begin{equation}\label{eq:psi2}
  \Psi_1(t)=m_1-\kappa t-\gamma t^2+o(t^2),\qquad t\downarrow0 .
\end{equation}
The third-order expansion of $\MR$ again follows from dominated convergence: writing
$e^z-1-z-\tfrac{z^2}{2}=z^3h_3(z)$ with
$h_3(z)=\tfrac12\int_0^1(1-u)^2e^{zu}\,du$ increasing in $z$, one dominates
$R^3h_3(tR)$ by $\{e^{t_0R}-1-t_0R-(t_0R)^2/2\}/t_0^3$ on $\{R\ge0\}$ for a fixed
$t_0\in(0,\tau)$ and by $|R|^3/6$ on $\{R<0\}$. Exactly as for $\kappa$, pure local
control \eqref{eq:purelocal} together with $\kappa=0$ forces $\gamma\ge0$, since
$\gamma<0$ would give $\Psi_1(t)>m_1$ for all small $t>0$.

\begin{theorem}[Square-root edge]\label{thm:edge}
Assume $\int_{|x|>1}x^4\nu(dx)<\infty$, $\tau>0$, $m_1>0$, $\kappa>0$ with $m_2$ and
$\kappa$ as in \eqref{eq:m2} and \eqref{eq:kappa}, pure local control
\eqref{eq:purelocal}, and $m_1>1/\tau$ when $\tau<\infty$. Then,
with $\varepsilon=\theta-1\downarrow0$,
\begin{equation}\label{eq:edge}
  \cstar(\theta V)=m_1-2\sqrt{\kappa\varepsilon}+o(\sqrt{\varepsilon}),
\end{equation}
and any maximiser $\hat t_\theta$ of $\Psi_\theta$ satisfies
$\hat t_\theta=\sqrt{\varepsilon/\kappa}\,(1+o(1))$. In particular $\cstar$ has one-sided
derivative $-\infty$ at $v=V$.
\end{theorem}

\begin{proof}
Write $\Psi_\theta(t)=\Psi_1(t)-\varepsilon\bigl(t\MR(t)\bigr)^{-1}$ and
$\Psi_1(t)=m_1-\kappa t+r(t)$ with $r(t)=o(t)$ as $t\downarrow0$. Since $\MR(t)=1+O(t)$
near the origin there are $\delta_0\in(0,\tau)$ and $C<\infty$ with
$\bigl|\varepsilon(t\MR(t))^{-1}-\varepsilon/t\bigr|\le C\varepsilon$ for
$0<t\le\delta_0$. Hence
\begin{equation}\label{eq:edgeexp}
  \Bigl|\Psi_\theta(t)-\Bigl(m_1-\kappa t-\frac{\varepsilon}{t}\Bigr)\Bigr|
  \;\le\;|r(t)|+C\varepsilon,
  \qquad 0<t\le\delta_0 .
\end{equation}

\emph{Lower bound on $\cstar$.} Evaluating \eqref{eq:edgeexp} at
$t=\sqrt{\varepsilon/\kappa}$, where $-\kappa t-\varepsilon/t=-2\sqrt{\kappa\varepsilon}$,
gives $\cstar(\theta V)\ge\Psi_\theta(t)\ge m_1-2\sqrt{\kappa\varepsilon}
-o(\sqrt\varepsilon)$, so
\begin{equation}\label{eq:edgeUB}
  \limsup_{\varepsilon\downarrow0}\frac{m_1-\cstar(\theta V)}{\sqrt\varepsilon}\;\le\;2\sqrt\kappa .
\end{equation}

\emph{Upper bound on $\cstar$.} Fix $\eta\in(0,\kappa)$ and choose
$\delta=\delta(\eta)\in(0,\delta_0)$, such that $|r(t)|\le\eta t$ for $t\le\delta$. Then
\eqref{eq:edgeexp} gives, for $0<t\le\delta$,
\[
  \Psi_\theta(t)\;\le\;m_1-(\kappa-\eta)t-\frac{\varepsilon}{t}+C\varepsilon
  \;\le\;m_1-2\sqrt{(\kappa-\eta)\varepsilon}+C\varepsilon .
\]
By \eqref{eq:away} of Lemma~\ref{lem:away} there is $\eta_\delta>0$ with
$\sup_{t\ge\delta}\Psi_1\le m_1-\eta_\delta$, and $\Psi_\theta\le\Psi_1$ pointwise. The
lower bound of the previous paragraph gives
$\sup_{t>0}\Psi_\theta\ge m_1-2\sqrt{\kappa\varepsilon}-o(\sqrt\varepsilon)$, which exceeds
$m_1-\eta_\delta$ for all small $\varepsilon$; hence for such $\varepsilon$ the supremum of
$\Psi_\theta$ is not approached on $[\delta,\infty)$ and
\[
  \cstar(\theta V)=\sup_{0<t\le\delta}\Psi_\theta(t)
  \;\le\;m_1-2\sqrt{(\kappa-\eta)\varepsilon}+C\varepsilon .
\]
Therefore $\liminf_{\varepsilon\downarrow0}(m_1-\cstar)/\sqrt\varepsilon\ge2\sqrt{\kappa-\eta}$
for every $\eta\in(0,\kappa)$, hence $\ge2\sqrt\kappa$. With \eqref{eq:edgeUB} this proves
\eqref{eq:edge}.

\emph{Existence of a maximiser.} By the previous paragraph the supremum for small
$\varepsilon$ is taken over the compact interval $[\varepsilon,\delta]$, on which
$\Psi_\theta$ is continuous, and by \eqref{eq:edgeexp} the values on $(0,\varepsilon)$ are
at most $m_1-1+(\eta+C)\varepsilon$, below the supremum. A maximiser therefore exists.

\emph{The maximiser.} Let $\hat t_\theta$ be any maximiser. By the previous paragraph
$\hat t_\theta\le\delta(\eta)$ for small $\varepsilon$, so
\[
  (\kappa-\eta)\hat t_\theta+\frac{\varepsilon}{\hat t_\theta}
  \;\le\;m_1+C\varepsilon-\Psi_\theta(\hat t_\theta)
  \;=\;2\sqrt{\kappa\varepsilon}+o(\sqrt\varepsilon).
\]
Writing $\hat t_\theta=s_\varepsilon\sqrt{\varepsilon/\kappa}$ and dividing by
$\sqrt{\kappa\varepsilon}$ yields
$\bigl(1-\eta/\kappa\bigr)s_\varepsilon+s_\varepsilon^{-1}\le2+o(1)$. Letting
$\varepsilon\downarrow0$ and then $\eta\downarrow0$ forces $s_\varepsilon\to1$, that is
$\hat t_\theta\sim\sqrt{\varepsilon/\kappa}$.

Finally, with $v=\theta V=V(1+\varepsilon)$,
\[
  \frac{\cstar(v)-\cstar(V)}{v-V}
  =\frac{-2\sqrt{\kappa\varepsilon}+o(\sqrt\varepsilon)}{V\varepsilon}
  \sim-\frac{2}{V}\sqrt{\frac{\kappa}{\varepsilon}}\;\longrightarrow\;-\infty,
\]
so the right derivative of $\cstar$ at $v=V$ is $-\infty$.
\end{proof}

At $\kappa=0$ the same method applies one order further and returns a cube-root law. This
settles the case left open by Theorem~\ref{thm:edge} under pure local control.

\begin{theorem}[Cube-root edge at the boundary case]\label{thm:cube}
Assume $\int_{|x|>1}|x|^5\nu(dx)<\infty$, $\tau>0$, $m_1>0$, pure local control
\eqref{eq:purelocal}, and $m_1>1/\tau$ when $\tau<\infty$. Assume further that $\kappa=0$
and $\gamma>0$, with $m_3$ and $\gamma$ as in \eqref{eq:m3} and \eqref{eq:gamma}. Then, with
$\varepsilon=\theta-1\downarrow0$,
\begin{equation}\label{eq:cube}
  \cstar(\theta V)=m_1-3\Bigl(\frac{\gamma\varepsilon^2}{4}\Bigr)^{1/3}
                   +o\bigl(\varepsilon^{2/3}\bigr),
\end{equation}
and any maximiser $\hat t_\theta$ of $\Psi_\theta$ satisfies
$\hat t_\theta=(\varepsilon/2\gamma)^{1/3}(1+o(1))$. In particular $\cstar$ again has
one-sided derivative $-\infty$ at $v=V$.
\end{theorem}

\begin{proof}
By \eqref{eq:psi2} with $\kappa=0$ write $\Psi_1(t)=m_1-\gamma t^2+r(t)$ with
$r(t)=o(t^2)$. As in the proof of Theorem~\ref{thm:edge} there are $\delta_0\in(0,\tau)$
and $C<\infty$ with $\bigl|\varepsilon(t\MR(t))^{-1}-\varepsilon/t\bigr|\le C\varepsilon$
for $0<t\le\delta_0$, so
\begin{equation}\label{eq:cubeexp}
  \Bigl|\Psi_\theta(t)-\Bigl(m_1-\gamma t^2-\frac{\varepsilon}{t}\Bigr)\Bigr|
  \;\le\;|r(t)|+C\varepsilon,\qquad 0<t\le\delta_0 .
\end{equation}
For $\gamma'>0$ the arithmetic--geometric mean inequality, applied to the three
terms $\gamma't^2$, $\varepsilon/(2t)$ and $\varepsilon/(2t)$, gives
\begin{equation}\label{eq:amgm3}
  \gamma't^2+\frac{\varepsilon}{t}\;\ge\;3\Bigl(\frac{\gamma'\varepsilon^2}{4}\Bigr)^{1/3},
  \qquad t>0,
\end{equation}
with equality exactly at $t=(\varepsilon/2\gamma')^{1/3}$.

\emph{Lower bound on $\cstar$.} Evaluating \eqref{eq:cubeexp} at
$t=(\varepsilon/2\gamma)^{1/3}$, where by \eqref{eq:amgm3} the bracket equals
$m_1-3(\gamma\varepsilon^2/4)^{1/3}$, and using $r(t)=o(t^2)=o(\varepsilon^{2/3})$ and
$C\varepsilon=o(\varepsilon^{2/3})$,
\begin{equation}\label{eq:cubeLB}
  \cstar(\theta V)\;\ge\;m_1-3\Bigl(\frac{\gamma\varepsilon^2}{4}\Bigr)^{1/3}
  -o\bigl(\varepsilon^{2/3}\bigr).
\end{equation}

\emph{Upper bound on $\cstar$.} Fix $\eta\in(0,\gamma)$ and choose
$\delta=\delta(\eta)\in(0,\delta_0)$ with $|r(t)|\le\eta t^2$ for $t\le\delta$. Then
\eqref{eq:cubeexp} and \eqref{eq:amgm3} give, for $0<t\le\delta$,
\[
  \Psi_\theta(t)\;\le\;m_1-(\gamma-\eta)t^2-\frac{\varepsilon}{t}+C\varepsilon
  \;\le\;m_1-3\Bigl(\frac{(\gamma-\eta)\varepsilon^2}{4}\Bigr)^{1/3}+C\varepsilon .
\]
By \eqref{eq:away} of Lemma~\ref{lem:away} there is $\eta_\delta>0$ with
$\sup_{t\ge\delta}\Psi_1\le m_1-\eta_\delta$, and $\Psi_\theta\le\Psi_1$ pointwise; by
\eqref{eq:cubeLB} the supremum of $\Psi_\theta$ exceeds $m_1-\eta_\delta$ for all small
$\varepsilon$, so for such $\varepsilon$ it is not approached on $[\delta,\infty)$ and
\[
  \cstar(\theta V)=\sup_{0<t\le\delta}\Psi_\theta(t)
  \;\le\;m_1-3\Bigl(\frac{(\gamma-\eta)\varepsilon^2}{4}\Bigr)^{1/3}+C\varepsilon .
\]
Hence $\liminf_{\varepsilon\downarrow0}(m_1-\cstar)\varepsilon^{-2/3}
\ge3\{(\gamma-\eta)/4\}^{1/3}$ for every $\eta\in(0,\gamma)$, and letting
$\eta\downarrow0$ and combining with \eqref{eq:cubeLB} proves \eqref{eq:cube}.

\emph{Existence of a maximiser.} On $(0,\varepsilon)$ one has $\varepsilon/t>1$ and
$|r(t)|\le\eta t^2\le\eta\varepsilon$ for $\varepsilon<1$, so \eqref{eq:cubeexp} bounds
$\Psi_\theta$ there by $m_1-1+(\eta+C)\varepsilon$, which is below the supremum for all
small $\varepsilon$. The supremum is therefore taken over the compact interval
$[\varepsilon,\delta]$, on which $\Psi_\theta$ is continuous, and is attained.

\emph{The maximiser.} Let $\hat t_\theta$ be any maximiser, so $\hat t_\theta\le\delta$
for small $\varepsilon$. By \eqref{eq:cubeexp} and \eqref{eq:cube},
\[
  (\gamma-\eta)\hat t_\theta^{\,2}+\frac{\varepsilon}{\hat t_\theta}
  \;\le\;m_1+C\varepsilon-\Psi_\theta(\hat t_\theta)
  \;=\;3\Bigl(\frac{\gamma\varepsilon^2}{4}\Bigr)^{1/3}+o\bigl(\varepsilon^{2/3}\bigr).
\]
Write $\hat t_\theta=s_\varepsilon(\varepsilon/2\gamma)^{1/3}$ and divide by
$(\gamma\varepsilon^2/4)^{1/3}$; since $\gamma(\varepsilon/2\gamma)^{2/3}
=(\gamma\varepsilon^2/4)^{1/3}$ and $\varepsilon(2\gamma/\varepsilon)^{1/3}
=2(\gamma\varepsilon^2/4)^{1/3}$, this reads
\[
  \Bigl(1-\frac{\eta}{\gamma}\Bigr)s_\varepsilon^2+\frac{2}{s_\varepsilon}\;\le\;3+o(1).
\]
The function $s\mapsto s^2+2/s$ has a strict minimum $3$ at $s=1$, so letting
$\varepsilon\downarrow0$ and then $\eta\downarrow0$ forces $s_\varepsilon\to1$, that is
$\hat t_\theta\sim(\varepsilon/2\gamma)^{1/3}$.

Finally $\{\cstar(\theta V)-\cstar(V)\}/(V\varepsilon)\sim
-3(\gamma/4)^{1/3}V^{-1}\varepsilon^{-1/3}\to-\infty$.
\end{proof}

\begin{remark}\label{rem:general-edge}
Theorems~\ref{thm:edge} and \ref{thm:cube} are the cases $j=1$ and $j=2$ of one
computation. If pure local control holds with $m_1>\max\{0,1/\tau\}$ and
$\Psi_1(t)=m_1-\gamma_jt^j+o(t^j)$ for some integer $j\ge1$ and some $\gamma_j>0$, then
the same three steps give
\[
  \cstar(\theta V)=m_1-(j+1)\,j^{-j/(j+1)}\bigl(\gamma_j\varepsilon^{\,j}\bigr)^{1/(j+1)}
  +o\bigl(\varepsilon^{j/(j+1)}\bigr),
  \qquad
  \hat t_\theta\sim\Bigl(\frac{\varepsilon}{j\gamma_j}\Bigr)^{1/(j+1)} ,
\]
so the one-sided derivative at $v=V$ is $-\infty$ for every such $j$. Only the
extremal-value step changes: $\gamma_jt^j+\varepsilon/t$ is minimised at
$t=(\varepsilon/j\gamma_j)^{1/(j+1)}$ with value
$(j+1)j^{-j/(j+1)}(\gamma_j\varepsilon^{j})^{1/(j+1)}$. The response is always slower than
linear in $\varepsilon$, which is what Corollary~\ref{cor:neveropt} uses.
\end{remark}

\begin{remark}
Theorem~\ref{thm:edge} explains the numerical behaviour reported in
Section~\ref{sec:examples}: a relaxation of the proxy by five percent lowers the pole by
roughly thirteen percent because the response is of order $\sqrt\varepsilon$, not
$\varepsilon$. It also shows that the variance-exact point is a singular point of the
frontier, which is what makes it a poor place to stand for tail bounds; see
Section~\ref{sec:noloss}.
\end{remark}

\subsection{Decay as the proxy grows}

\begin{theorem}[Two-sided threshold bounds]\label{thm:asym}
Let $t_\theta$ be as in Proposition~\ref{prop:tth}. For every $\lambda>1$,
\begin{equation}\label{eq:sandwich}
  \Bigl(1-\frac1\lambda\Bigr)\frac{1}{t_{\lambda\theta}}
  \;\le\;\cstar(\theta V)\;\le\;\frac{1}{t_\theta}.
\end{equation}
Consequently, if
\begin{equation}\label{eq:ratio}
  \lim_{\theta\to\infty}\frac{t_{\lambda\theta}}{t_\theta}=1
  \qquad\text{for every fixed }\lambda>1,
\end{equation}
then $\cstar(\theta V)\sim1/t_\theta$ as $\theta\to\infty$.
\end{theorem}

\begin{proof}
The upper bound is Proposition~\ref{prop:tth}. For the lower bound take
$t=t_{\lambda\theta}$, the bound being trivial if $t=\infty$; if $t=0$ then $\tau=0$, so
$\cstar\equiv\infty$ by Proposition~\ref{prop:degenerate} and there is nothing to prove. If $t<\tau$ then $\MR$ is
continuous at $t$, so $\MR(t)\ge\lambda\theta$ and
$\Psi_\theta(t)\ge t^{-1}(1-1/\lambda)$; if $t=\tau$ then Proposition~\ref{prop:floor}
gives the stronger bound $\cstar(\theta V)\ge1/\tau=1/t$. Under \eqref{eq:ratio}, dividing
\eqref{eq:sandwich} by $1/t_\theta$ and letting $\theta\to\infty$ gives
$1-1/\lambda\le\liminf\cstar t_\theta\le\limsup \cstar t_\theta\le1$ for every
$\lambda>1$.
\end{proof}

\begin{corollary}[Bounded positive jumps]\label{cor:bounded}
Suppose the positive part of $\nu$ has bounded support with radius
\[
  b=\sup\{x>0:x\in\operatorname{supp}\nu\}\in(0,\infty).
\]
Then $\log\MR(t)/t\to b$, condition \eqref{eq:ratio} holds, and
\begin{equation}\label{eq:blog}
  \cstar(v)\;\sim\;\frac{b}{\log(v/V)}\qquad (v\to\infty).
\end{equation}
\end{corollary}

\begin{proof}
First, $\tau=\infty$: for $0<x\le b$ one has $e^{tx}-1-tx\le\tfrac12t^2x^2e^{tb}$ and for
$x<0$, $e^{tx}-1-tx\le\tfrac12t^2x^2$, so $K_X(t)<\infty$ for every $t>0$ by the standing
assumption $\int x^2\nu(dx)<\infty$. Since $B\in(0,1)$ and $Z\le b$ we have $R\le b$
almost surely, so $\MR(t)\le e^{bt}$. For
$\varepsilon>0$, $Z$ charges $(b-\varepsilon/2,b]$ and $B$ has positive density near $1$,
so $\PP(R>b-\varepsilon)>0$ and
$\MR(t)\ge e^{(b-\varepsilon)t}\PP(R>b-\varepsilon)$. Hence $\log\MR(t)/t\to b$ and
$t_\theta=(\log\theta)/b\,(1+o(1))$, which gives both \eqref{eq:ratio} and
\eqref{eq:blog}.
\end{proof}

\begin{remark}\label{rem:slow}
The convergence in \eqref{eq:blog} is slow, because, when the law of $R$ has a density vanishing to a finite order at $b$,
$\log\MR(t)=bt-k\log t+O(1)$ with $k$ that order plus one. In
numerical work the implicit relation $bt-k\log t=\log\theta+O(1)$ should be solved for
$t_\theta$ rather than using \eqref{eq:blog}, and the constant term should be retained:
dropping it costs about half of the gain. See Example~\ref{ex:twopoint}, where at
$v/V=10^8$ the crude form overstates the pole by more than a factor of one and a half,
the implicit form is accurate to about twenty percent, and the implicit form without the
constant term to about forty percent. The same phenomenon appears
in the rare-jump asymptotics of \cite{ChenWang2026a}.
\end{remark}

\section{Directional and multivariate frontiers}\label{sec:multi}

Nothing above used a sign restriction on $\nu$, so the results apply verbatim to the two
directional frontiers of a general centered infinitely divisible law. Writing
$\cstar^+(v)=\cstar(v)$ and $\cstar^-(v)=\cstar(v;-X)$, Theorem~\ref{thm:var} reads
\[
  \cstar^{\pm}(v)=\max\Bigl\{0,\ \sup_{t>0}\frac1t\Bigl(1-\frac{\theta}{\E e^{\pm tR}}\Bigr)\Bigr\},
\]
and Propositions~\ref{prop:floor} and \ref{prop:zero} give the corresponding floors
$1/\tau_\pm$, where $\tau_\pm=\sup\{t>0:\E e^{\pm tX}<\infty\}$, together with the zero
pattern of \cite{ChenWang2026a}: $\cstar^{+}$ vanishes identically exactly when
$\nu((0,\infty))=0$, and $\cstar^{-}$ exactly when $\nu((-\infty,0))=0$. A law can have a constant frontier in one direction and a strictly decreasing one in the
other: for the one-sided tempered stable law of Example~\ref{ex:ts} with $Y<-1$,
Proposition~\ref{prop:zero} applied to $-X$ gives $\cstar^-\equiv0$, while
Proposition~\ref{prop:ts} makes $\cstar^+$ drop strictly to the right of $V$. For two-sided laws the two
directions are not independent, since by \eqref{eq:R} the whole Kolmogorov measure enters
$\MR$ in each of them. The
bilateral Gamma family of Example~\ref{ex:gamma} is constant in both.

For a centered infinitely divisible random vector $X$ in $\R^d$ with covariance $\Sigma$
and $V_X(t)=t^\top\Sigma t$, fix an inflation factor $\theta\ge1$ and define
$C_X(t;\theta)$ as the least $c\ge0$ with
\[
  K_X(st)\le\frac{\theta\,V_X(t)\,s^2}{2(1-cs)},\qquad 0\le s<1/c .
\]
For each $t$ with $V_X(t)>0$ the projection $\langle t,X\rangle$ is itself a centered
infinitely divisible scalar variable with variance $V_X(t)$, so Lemma~\ref{lem:fact}
applied to it supplies a variable $R_t$ with $K_X(st)=V_X(t)s^2\,\E e^{sR_t}/2$; this is
the directional factorization of \cite{ChenWang2026b}. Hence
Theorem~\ref{thm:var} applies in each direction and
\[
  C_X(t;\theta)=\max\Bigl\{0,\ \sup_{s>0}\frac1s\Bigl(1-\frac{\theta}{\E e^{sR_t}}\Bigr)\Bigr\}.
\]
When $V_X(t)=0$ the projection $\langle t,X\rangle$ vanishes almost surely, so
$K_X(st)\equiv0$ and the least admissible $c$ is $0$; we read $C_X(t;\theta)=0$ there. For
fixed $\theta$ the map $t\mapsto C_X(t;\theta)$ is positively homogeneous, and it is the
pointwise least such denominator compatible with the inflated quadratic form
$\theta\,t^\top\Sigma t$, in the following sense; the argument is that of
\cite[Theorem 2.8]{ChenWang2026b}, and we give it since it is short. Call $q:\R^d\to[0,\infty]$ admissible if it is positively homogeneous and
$K_X(t)\le\theta V_X(t)/\{2(1-q(t))\}$ whenever $q(t)<1$. Then $C_X(\cdot;\theta)$ is
admissible, by taking $s=1$ in its defining envelope; and if $q$ is admissible and $t$ is
fixed, then for $0\le s<1/q(t)$ homogeneity gives $q(st)=sq(t)<1$, so applying the
displayed inequality at $st$ yields $K_X(st)\le\theta V_X(t)s^2/\{2(1-sq(t))\}$. Hence
$q(t)$ is a feasible directional scale and $q(t)\ge C_X(t;\theta)$.
Thus the radial pole is not a single function but a one-parameter family of them, and
Theorem~\ref{thm:convex} describes how each direction moves as
$\theta$ grows.

\section{Examples}\label{sec:examples}

Throughout this section frontiers are computed from \eqref{eq:varformula} on a fine
logarithmic grid in $t$, refined by golden-section search near the maximiser; all displayed
digits were stable under refinement. One practical warning: near the origin the quotient
$\MR=2K_X/(Vt^2)$ suffers severe cancellation in floating point, and the power series of
$\MR$ must be used there instead.

\begin{example}[Gaussian]
If $\nu=0$ then $R=0$ almost surely, $\MR\equiv1$, and $\Psi_\theta\le0$; hence
$\cstar\equiv0$. By Proposition~\ref{prop:zero} the frontier vanishes identically for
exactly the laws with $\nu((0,\infty))=0$, of which this is one.
\end{example}

\begin{example}[Gamma and bilateral Gamma]\label{ex:gamma}
Let $X$ be a centered $\mathrm{Gamma}(\alpha,\text{scale }\beta)$ variable, so
$K_X(t)=\alpha\{-\log(1-\beta t)-\beta t\}$ and $\tau=1/\beta$. Its L\'evy measure is
$\nu(dx)=\alpha x^{-1}e^{-x/\beta}dx$, the case $Y=0$, $C=\alpha$, $M=1/\beta$ of
Example~\ref{ex:ts}, so Proposition~\ref{prop:ts} below gives $\cstar(V)=1/M=\beta$ and
boundary control, recovering the value obtained in \cite{ChenWang2026a}. By
Proposition~\ref{prop:floor} the frontier is then constant:
\[
  \cstar(v)=\beta\qquad\text{for every }v\ge V .
\]
Numerically the computed frontier equals $\beta$ to six digits for $v/V$ ranging over
$[1,10^3]$. The same holds for each direction of a centered bilateral Gamma law
\cite{KuchlerTappe2008}, whose directional poles are the two scale parameters. Indeed, for
the right direction write $\MR=(V_+/V)\MR^{(+)}+(V_-/V)\MR^{(-)}$ for the decomposition of
$H_X$ into its two sides; the negative side has $Z\le0$, so $\MR^{(-)}(t)\le1$ for $t>0$,
while $\MR^{(+)}(t)\le(1-\beta_+t)^{-1}$ by the one-sided case. Since
$(1-\beta_+t)^{-1}\ge1$, the convex combination obeys the same bound, so $c=\beta_+$ is
feasible; Proposition~\ref{prop:floor} supplies the matching lower bound. Relaxing the quadratic proxy is useless
for this family.
\end{example}

\begin{example}[Compound Poisson with Gamma jumps]\label{ex:g4}
Let $X=\sum_{i\le N}J_i-\lambda\E J$ with $N\sim\mathrm{Poisson}(\lambda)$ and
$J\sim\mathrm{Gamma}(\alpha,\text{scale }\beta)$. Then $\tau=1/\beta$. This law is the case
$Y=-\alpha$, $M=1/\beta$ of Example~\ref{ex:ts}, so Proposition~\ref{prop:ts} gives
$\cstar(V)=\beta\max\{1,(\alpha+2)/3\}$, the value of \cite{Chen2026a}, and for $\alpha>1$ it
supplies pure local control \eqref{eq:purelocal} together with $m_1>1/\tau$; note that
$\cstar(V)=m_1$ alone is weaker than \eqref{eq:purelocal}, which additionally requires the
supremum not to be attained at any $t>0$. Take $\alpha=4$, $\beta=1$, so $\cstar(V)=2$, $m_1=2$, $m_2=7$ and
$\kappa=m_1^2-m_2/2=1/2$. The frontier is given in Table~\ref{tab:g4}.

\begin{table}[t]
\centering
\caption{Frontier $\cstar(v)$ and its maximiser for the centered compound Poisson law with
$\mathrm{Gamma}(4,1)$ jumps.}
\label{tab:g4}
\begin{tabular}{lcl}
\toprule
$v/V$ & $\cstar(v)$ & maximiser \\
\midrule
$1$      & $2.0000$ & $t\downarrow0$ (local) \\
$1.01$   & $1.8701$ & $0.127$ \\
$1.05$   & $1.7380$ & $0.251$ \\
$1.5$    & $1.4294$ & $0.538$ \\
$3$      & $1.2536$ & $0.709$ \\
$20$     & $1.1001$ & $0.876$ \\
$10^{4}$ & $1.0104$ & $0.986$ \\
$10^{6}$ & $1.0022$ & $0.997$ \\
\bottomrule
\end{tabular}
\end{table}

\noindent
The maximiser leaves the origin immediately, as Proposition~\ref{prop:knife} requires, and
the values decrease towards the floor $1/\tau=1$ of Proposition~\ref{prop:floor}, with the
decreasing convergence of Proposition~\ref{prop:vstar}; numerically the floor is not
reached at any $v/V\le10^{6}$. The square-root law of
Theorem~\ref{thm:edge} predicts $\cstar\approx2-2\sqrt{\varepsilon/2}$; the ratio of the
observed drop to the predicted one is $0.919$ at $\varepsilon=10^{-2}$, $0.974$ at
$10^{-3}$, $0.992$ at $10^{-4}$ and $0.997$ at $10^{-5}$, and the observed maximiser
agrees with $\sqrt{2\varepsilon}$ to within $0.4\%$ at $\varepsilon=10^{-5}$
($4.456\times10^{-3}$ against $4.472\times10^{-3}$).
\end{example}

\begin{example}[One-sided tempered stable and CGMY]\label{ex:ts}
Let $\nu(dx)=Cx^{-1-Y}e^{-Mx}dx$ on $(0,\infty)$ with $Y<2$, $C,M>0$
\cite{CGMY2002,Rosinski2007}, so that $\kappa_n=C\Gamma(n-Y)M^{Y-n}$ for $n\ge2$ and
$\tau=M$. Here and below, $Y$ denotes the tempered stable index in the CGMY convention,
which is unrelated to the canonical variable $Z$ of \eqref{eq:R}. Then
\begin{equation}\label{eq:tsmoments}
  \E R^n=\frac{2\Gamma(n+2-Y)}{(n+1)(n+2)\Gamma(2-Y)}\,M^{-n},
  \qquad n\ge1 .
\end{equation}
Proposition~\ref{prop:ts} below determines the whole frontier for this family.
\end{example}

\begin{proposition}\label{prop:ts}
For the law above,
\begin{equation}\label{eq:tspole}
  \cstar(V)=\frac1M\max\Bigl\{1,\ \frac{2-Y}{3}\Bigr\},
\end{equation}
with local control for $Y\le-1$ and boundary control for $Y\ge-1$ in the sense of
Definition~\ref{def:control}. Moreover the
frontier is constant for $Y\ge-1$, while for $Y<-1$ it decreases from \eqref{eq:tspole} to
the floor $1/M$.
\end{proposition}

\begin{proof}
Put $m=(2-Y)/(3M)$, which is $\E R$ by \eqref{eq:tsmoments}. Comparing successive moments
in \eqref{eq:tsmoments} with those of an exponential law of mean $m$, the inequality
$\E R^n\le n!\,m^n$ propagates from $n$ to $n+1$ precisely when
\[
  \frac{n+2-Y}{n+3}\le\frac{2-Y}{3}
  \iff 3(n+2-Y)\le(2-Y)(n+3)
  \iff 3n\le 2n-Yn
  \iff Y\le-1 ,
\]
for every $n\ge1$; the induction starts from the base case $\E R=m$, which is an equality.
Here $R\ge0$ because $\nu$ is carried by $(0,\infty)$, so all
moments are nonnegative and
\begin{equation}\label{eq:momdom}
  \E R^n\le n!\,c^n\ \ \text{for all }n\ge1
  \quad\Longrightarrow\quad
  \MR(t)=\sum_{n\ge0}\frac{\E R^n}{n!}t^n\le\sum_{n\ge0}(ct)^n=\frac{1}{1-ct},
  \qquad 0\le t<1/c,
\end{equation}
by monotone convergence. For $Y\le-1$ this gives feasibility of $m$, and the local bound
$\cstar(V)\ge m$ of Proposition~\ref{prop:knife} gives equality.

For $Y\ge-1$ note first that at $Y=-1$ formula \eqref{eq:tsmoments} reduces to
\[
  \E R^n=\frac{2\,\Gamma(n+3)}{(n+1)(n+2)\Gamma(3)}M^{-n}
        =\frac{(n+2)!}{(n+1)(n+2)}M^{-n}=n!\,M^{-n},
\]
so $R$ has exactly the moments of an exponential variable of mean $1/M$. For general
$Y<2$,
\[
  \frac{\Gamma(n+2-Y)}{\Gamma(2-Y)}=\prod_{j=0}^{n-1}(2-Y+j),
\]
a product of positive factors each nonincreasing in $Y$; hence $\E R^n$ is nonincreasing
in $Y$ and $\E R^n\le n!\,M^{-n}$ for every $Y\ge-1$. By \eqref{eq:momdom} the value $1/M$
is feasible, and Proposition~\ref{prop:floor} gives $\cstar(V)\ge1/\tau=1/M$; the two
bounds agree.

For the frontier: when $Y\ge-1$ the pole is boundary controlled, so
Proposition~\ref{prop:floor} makes the frontier constant. When $Y<-1$ the ratio condition
above is \emph{strict} at $n=1$, so $\E R^2<2m^2$ and hence, by the same termwise
comparison, $\MR(t)<(1-mt)^{-1}$ for $0<t<1/m$; note $1/m<M=\tau$ because $m>1/M$. This
gives $\Psi_1(t)<m$ for $0<t<1/m$, while $\Psi_1(t)\le1/t<m$ for $t>1/m$; at the remaining
point $t=1/m$ one has $\MR(1/m)<\infty$ because $1/m<\tau$, so
$\Psi_1(1/m)=m\{1-1/\MR(1/m)\}<m$ as well. Thus pure
local control \eqref{eq:purelocal} holds; with $m>1/M=1/\tau$ this is the hypothesis of
Lemma~\ref{lem:away}, so $\cstar(V)=m>\max\{0,1/\tau\}$ and Theorem~\ref{thm:drop}
gives a strict decrease to the right of $V$; Propositions~\ref{prop:floor} and
\ref{prop:vstar} give the floor $1/M$ and the decreasing convergence to it. The
intermediate case $Y=-1$ is excluded from the strict statement: there $R$ is exactly
exponential with mean $1/M$, $\Psi_1\equiv1/M$ on $(0,\tau)$, and the frontier is again
constant.
\end{proof}

\noindent
Formula \eqref{eq:tspole} contains two known cases. Taking $Y=-\alpha$ recovers the
compound Poisson Gamma-jump pole $\beta\max\{1,(\alpha+2)/3\}$ of \cite{Chen2026a} with
$\beta=1/M$, and taking $Y=0$ gives the Gamma scale $\beta$ used in
Example~\ref{ex:gamma}. Formula
\eqref{eq:tspole} does \emph{not} extend to bilateral laws: there $R$ is built from the
whole Kolmogorov measure, so the jumps of the opposite sign also enter $\MR$ and lower it.
For the bilateral law with $\nu(dx)=Cx^{-1-Y_+}e^{-Mx}dx$ on $(0,\infty)$ and
$C|x|^{-1-Y_-}e^{-G|x|}dx$ on $(-\infty,0)$, taking $C=G=M=1$ and $Y_+=Y_-=-3$ gives right
pole $1.1914$, not $(2-Y_+)/3=5/3$. Numerically, with
$M=1$, the frontier at $v/V\in\{1,1.5,3,10,10^3\}$ is $1.667,1.251,1.132,1.059,1.005$ for
$Y=-3$ and $1.333,1.089,1.034,1.009,1.000$ for $Y=-2$, and is constant equal to $1$ for
$Y\in\{-1,0,1\}$, in agreement with Proposition~\ref{prop:ts}.

\begin{example}[Skellam and the disappearance of a phase transition]\label{ex:skellam}
Let $\nu=\lambda\{p\delta_1+(1-p)\delta_{-1}\}$ and $V=\lambda$, the centered Skellam law \cite{Skellam1946}.
Here $\tau=\infty$ and, with $d=2p-1$,
\[
  \MR(t)=\frac{2}{t^2}\bigl\{\cosh t-1+d(\sinh t-t)\bigr\}.
\]
At $v=V$ the right pole equals $d/3$ exactly when $p\ge p_+=(2+\sqrt3)/4$, and is attained
at a positive finite argument otherwise; this is Lemma~\ref{lem:skellam} below, and it
recovers the transition of \cite{ChenWang2026a}. Proposition~\ref{prop:knife} shows
that this transition cannot survive any relaxation of the proxy: for every $\theta>1$ the
objective diverges at the origin, so control is interior for every $p>0$. The computed
frontier confirms this. For $p=1$, where the law is a centered Poisson variable, the pole
falls from $1/3$ at $\theta=1$ to $0.3231$ at $\theta=1.001$ with maximiser $0.19$, and to
$0.2709$ at $\theta=1.05$; at $p=p_+$ the corresponding values are $0.2887$, $0.2850$ and
$0.2516$. The two columns obey different edge laws. At $p=1$ one has $\kappa=1/36>0$ and
Theorem~\ref{thm:edge} applies. At $p=p_+$ one has $\kappa=0$, and \eqref{eq:gamma} with
$\kappa_3=\kappa_5=\lambda d$, $\kappa_4=\lambda$ and $V=\lambda$ gives
\[
  \gamma=\frac{d}{18}-\frac{d^3}{27}-\frac{d}{60}=0.0096225>0
  \qquad\text{at } d=\tfrac{\sqrt3}{2},
\]
so Theorem~\ref{thm:cube} applies instead and predicts a drop $3(\gamma\varepsilon^2/4)^{1/3}$
with maximiser $(\varepsilon/2\gamma)^{1/3}$. Numerically the ratio of the observed drop to
the predicted one is $0.910$ at $\varepsilon=10^{-3}$, $0.958$ at $10^{-4}$, $0.991$ at
$10^{-6}$ and $0.998$ at $10^{-8}$, and the corresponding ratios for the maximiser are
$1.028$, $1.013$, $1.003$ and $1.001$. The transition at $p_+$ is where $\kappa$ changes sign: here $m_1=d/3$ and $m_2=1/6$, so
$\kappa=d^2/9-1/12$ vanishes at $d=\pm\sqrt3/2$, that is at $p_\pm=(2\pm\sqrt3)/4$. Only
$p_+$ is a transition for the right pole: near $p_-$ one has $m_1=d/3<0$, so
$\cstar(V)>\max\{0,m_1\}$ and the maximiser is interior on both sides of it, and $p_-$ is
the corresponding threshold for the opposite direction. Both are features of the single
point $v=V$, not of the family.
\end{example}

\begin{lemma}[The variance-exact Skellam threshold]\label{lem:skellam}
Let $0<p\le1$ and $d=2p-1$ in Example~\ref{ex:skellam}. Then $\cstar(V)=d/3$ if
$d\ge\sqrt3/2$, and $\cstar(V)>\max\{0,d/3\}$ if $d<\sqrt3/2$; in the latter case the
supremum in \eqref{eq:varformula} is attained at a positive finite argument. (At $p=0$
there are no positive jumps and $\cstar\equiv0$ by Proposition~\ref{prop:zero}.)
\end{lemma}

\begin{proof}
Here $Z=\pm1$ with probabilities $p,1-p$ and $\E B^n=2/\{(n+1)(n+2)\}$, so
\[
  \E R^n=\frac{2}{(n+1)(n+2)}\times
  \begin{cases}1,&n\text{ even},\\ d,&n\text{ odd},\end{cases}
  \qquad m_1=\E R=\frac d3 .
\]
Let $d\ge\sqrt3/2$. We claim $\E R^n\le n!\,m_1^n$ for every $n\ge1$. At $n=1$ both sides
equal $d/3$. For $n\ge2$ use $|d|\le1$ to get $\E R^n\le2/\{(n+1)(n+2)\}$, so it suffices
that $b_n:=n!\,3^{n/2}(n+1)(n+2)/(2\cdot6^n)\ge1$; and indeed $b_2=1$ while
$b_{n+1}/b_n=\sqrt3\,(n+3)/6$, which is $5\sqrt3/6>1$ at $n=2$ and increases in $n$. Here
$R$ takes both signs, so the passage from moments to \eqref{eq:momdom} rests on absolute
rather than monotone convergence; that is immediate, since $|R|=B|Z|\le1$ almost surely
gives $\E e^{t|R|}\le e^{t}<\infty$ and the exponential series may be summed termwise.
Hence
$\MR(t)\le(1-m_1t)^{-1}$ on $(0,1/m_1)$, the scale $m_1$ is feasible, and
$\cstar(V)\le m_1$; with Proposition~\ref{prop:knife} this is an equality.

Now let $d<\sqrt3/2$. If $d>0$ then $\kappa=m_1^2-m_2/2=d^2/9-1/12<0$, since $m_2=\E R^2=1/6$
and $0<d<\sqrt3/2$; by the expansion $\Psi_1(t)=m_1-\kappa t+o(t)$ of
Section~\ref{sec:asym} we then have $\Psi_1(t)>m_1>0$ for all small $t>0$, whence
$\cstar(V)>m_1=\max\{0,m_1\}$. If $d\le0$ then $\max\{0,m_1\}=0$, while $p>0$ gives
$\nu((0,\infty))>0$ and Proposition~\ref{prop:zero} yields $\cstar(V)>0$. In both cases
the supremum is attained: here $\tau=\infty$, so $\Psi_1$ is continuous on $(0,\infty)$ by
Lemma~\ref{lem:psicont}, tends to $m_1$ at the origin and satisfies $\Psi_1(t)\le1/t\to0$
as $t\to\infty$, and a supremum strictly above $\max\{0,m_1\}$ must therefore be attained
on a compact subinterval.
\end{proof}

\begin{example}[Bounded two-point jumps]\label{ex:twopoint}
Let $X$ be centered compound Poisson with $\PP(J=1)=39/40$ and $\PP(J=5)=1/40$, so
$b=5$ and $\tau=\infty$. Formula \eqref{eq:varformula} gives $\cstar(V)=0.895416$, attained
at the interior point $t_-=0.473723$ used in Section~\ref{sec:noloss}; this is the value
of \cite{Chen2026a}. The intensity $\lambda$ is left free: since $K_X$ and $V$
are both proportional to $\lambda$, the ratio $\MR=2K_X/(Vt^2)$ does not depend on it, and
neither does any quantity displayed in this example. Section~\ref{sec:noloss} fixes
$\lambda=1$, where deviation levels enter and the scale does matter. The frontier is given in Table~\ref{tab:twopoint}.

\begin{table}[t]
\centering
\caption{Frontier $\cstar(v)$ for the bounded two-point jump law of
Example~\ref{ex:twopoint}.}
\label{tab:twopoint}
\begin{tabular}{lccccc}
\toprule
$v/V$   & $2$ & $10$ & $10^2$ & $10^4$ & $10^8$\\
\midrule
$\cstar$ & $0.6588$ & $0.4855$ & $0.3707$ & $0.2603$ & $0.1679$\\
\bottomrule
\end{tabular}
\end{table}
The density of $R$ vanishes linearly at $b$, so $\log\MR(t)=5t-2\log t-\log32+o(1)$ and
Remark~\ref{rem:slow} applies with $k=2$. At $v/V=10^8$ the crude form $b/\log(v/V)$ gives
$0.2714$ and the implicit form $1/t_\theta$ gives $0.1991$, against the true value
$0.1679$; the ratio $\cstar t_\theta$ is $0.84$ and increases to $1$ only logarithmically.
\end{example}

\section{Optimising over the frontier}\label{sec:noloss}

We now show that the frontier, rather than any single point of it, is the object that
governs concentration.

Recall the standard Legendre computation for the envelope
$g_{v,c}(t)=vt^2/\{2(1-ct)\}$ \cite[Section~2.4]{BLM2013}: for $y>0$,
\begin{equation}\label{eq:legendre}
  \sup_{0\le t<1/c}\bigl\{ty-g_{v,c}(t)\bigr\}=x
  \qquad\Longleftrightarrow\qquad
  y=\sqrt{2vx}+cx .
\end{equation}
Consequently a feasible pair $(v,c)$ yields the Bernstein bound
$\PP\{X\ge\sqrt{2vx}+cx\}\le e^{-x}$, and the best deviation obtainable from the whole
frontier at confidence level $e^{-x}$ is
\begin{equation}\label{eq:ystar}
  y_*(x)=\inf_{v\ge V}\bigl\{\sqrt{2vx}+\cstar(v)\,x\bigr\}.
\end{equation}
Write $I(y)=\sup_{t}\{ty-K_X(t)\}$ for the Cram\'er transform of $X$, so that the exact
Chernoff bound is $\PP\{X\ge y\}\le e^{-I(y)}$.

Because $I$ need not be surjective onto $(0,\infty)$ we use the generalized inverse
\begin{equation}\label{eq:geninv}
  I^{-1}(x)=\inf\{y>0:I(y)\ge x\},
\end{equation}
and for a feasible pair $(v,c)$ we write $I_{g}(y)=\sup_{0\le t<1/c}\{ty-g_{v,c}(t)\}$,
the supremum being restricted to the domain of the envelope. Define the \emph{tangency
set}
\begin{equation}\label{eq:Tset}
  \mathcal T=\Bigl\{t>0:\ \exists\,v\ge V \text{ with } c=\cstar(v)<\infty,\ t<1/c,\
  K_X(t)=g_{v,c}(t)\Bigr\}\subseteq(0,\tau).
\end{equation}
The inclusion holds because Proposition~\ref{prop:floor} gives $c\ge1/\tau$, so
$t<1/c\le\tau$. When $c>0$, tangency at $t$ is by \eqref{eq:psi} the same as
$\Psi_{v/V}(t)=c$, that is,
$t$ attains the supremum in \eqref{eq:varformula}; when $c=0$ it says that
$K_X(t)=vt^2/2$, which is a strictly stronger requirement than attaining that supremum
(Remark~\ref{rem:whyhyp}).

\begin{theorem}[No loss over the frontier]\label{thm:noloss}
Assume $\tau>0$ and let $x>0$.
\begin{enumerate}
\item[(i)] $y_*(x)\ge I^{-1}(x)$.
\item[(ii)] Let $t_v\in\mathcal T$, witnessed by $v\ge V$ with $c=\cstar(v)$. Put
$y=K_X'(t_v)$ and $x=I(y)$. Then
\[
  \sqrt{2vx}+\cstar(v)\,x\;=\;y\;=\;I^{-1}(x),
\]
so the infimum in \eqref{eq:ystar} is attained at this $v$ and equals the exact Chernoff
deviation.
\item[(iii)] Consequently $y_*=I^{-1}$ on the set of levels
$\{I(K_X'(t)):t\in\mathcal T\}$.
\item[(iv)] Conversely, if $y_*(x)=I^{-1}(x)$ then $K_X$ has a Legendre maximiser at $y=I^{-1}(x)$ and
it belongs to $\mathcal T$. Equality at level $x$ therefore holds exactly when a Legendre
maximiser at that level exists and lies in $\mathcal T$.
\end{enumerate}
The infimum in \eqref{eq:ystar} is attained: $v\mapsto\sqrt{2vx}+\cstar(v)x$ is finite and
continuous on $[V,\infty)$ by Propositions~\ref{prop:degenerate} and \ref{prop:cont}, and
tends to infinity.
\end{theorem}

\begin{proof}
(i) Let $v\ge V$ with $c=\cstar(v)<\infty$ and put $y=\sqrt{2vx}+cx$, so that $I_g(y)=x$
by \eqref{eq:legendre}. Feasibility gives $K_X\le g_{v,c}$ on $[0,1/c)$, hence
\[
  I(y)=\sup_{t\in\R}\{ty-K_X(t)\}\;\ge\;\sup_{0\le t<1/c}\{ty-K_X(t)\}\;\ge\;I_g(y)=x,
\]
so $y\ge I^{-1}(x)$ by \eqref{eq:geninv}. Taking the infimum over $v$ gives (i).

(ii) By \eqref{eq:Tset} we have $t_v<1/c$ and $K_X(t_v)=g_{v,c}(t_v)$; since
$K_X\le g_{v,c}$ on $[0,1/c)$ with equality at the interior point $t_v$, the two functions
are in fact tangent there. Since $K_X$ is finite on the open interval
$(0,\tau)$ it is real analytic there, so $y=K_X'(t_v)$ is defined; the concave function
$t\mapsto ty-K_X(t)$ has vanishing derivative at the interior point $t_v$, whence
$I(y)=t_vy-K_X(t_v)=x$. Moreover
\[
  I_g(y)\;\ge\;t_vy-g_{v,c}(t_v)\;=\;t_vy-K_X(t_v)\;=\;x,
\]
while the display in (i) gives $I_g(y)\le I(y)=x$; so $I_g(y)=x$ and
\eqref{eq:legendre} yields $y=\sqrt{2vx}+cx$. Finally $I$ is convex with $I(0)=0$, because
$K_X\ge0$ and $K_X(0)=0$; hence for $0<y'<y$ we have $I(y')\le(y'/y)I(y)<x$, so
$I^{-1}(x)=y$. With (i) this shows the infimum is attained at $v$.

(iii) is (ii) applied at each $t\in\mathcal T$.

(iv) Suppose $y_*(x)=I^{-1}(x)$, and let $v$ attain the infimum, with $c=\cstar(v)<\infty$
and $y=\sqrt{2vx}+cx=I^{-1}(x)$. The map $t\mapsto ty-g_{v,c}(t)$ is continuous on
$[0,1/c)$ and tends to $-\infty$ at the right endpoint when $c>0$, and is a concave
quadratic when $c=0$; so it attains $I_g(y)=x$ at some $t_g$, and $t_g>0$ because $y>0$.
Next, $I^{-1}(x)=y$ gives $I(y')<x$ for every $y'<y$, and $I$ is lower semicontinuous
being a supremum of affine functions, so $I(y)\le\liminf_{y'\uparrow y}I(y')\le x$. Hence
\[
  x=t_gy-g_{v,c}(t_g)\;\le\;t_gy-K_X(t_g)\;\le\;I(y)\;\le\;x,
\]
using $K_X\le g_{v,c}$. All the inequalities are equalities, so $K_X(t_g)=g_{v,c}(t_g)$
and $t_g$ is a Legendre maximiser for $K_X$ at $y$. Since $t_g<1/c$, this places
$t_g\in\mathcal T$ by \eqref{eq:Tset}. The maximiser is unique: $X$ is nondegenerate, so
$K_X''>0$ on $(0,\tau)$ and $t\mapsto ty-K_X(t)$ is strictly concave there.

For the converse, suppose the Legendre maximiser $t^*$ at $y=I^{-1}(x)$ exists and lies in
$\mathcal T$. Then $K_X'(t^*)=y$ with $t^*<\tau$, so $0<y<\lim_{s\uparrow\tau}K_X'(s)$, and
on that range $I$ is finite, hence continuous, being convex. Now $I$ is nondecreasing on
$[0,\infty)$, so $I^{-1}(x)=y$ gives $I(y')<x$ for $y'<y$ and $I(y')\ge x$ for $y'>y$;
continuity forces $I(y)=x$. Part (ii) applied to $t^*$ then gives
$y_*(x)=y=I^{-1}(x)$. Hence equality at level $x$ occurs precisely when the Legendre
maximiser at that level lies in $\mathcal T$.
\end{proof}

\begin{remark}\label{rem:whyhyp}
The restriction to $c>0$ above cannot be dropped: when $\cstar(v)=0$ the identity
$c=\Psi_{v/V}(t_v)$ fails, because the maximum with $0$ in \eqref{eq:varformula} binds, and
attaining the supremum of $\Psi_{v/V}$ is then no longer the same as tangency. For
instance, let $X$ be centered compound Poisson with all jumps equal to $-1$ and intensity
$1$, so $V=1$, $\tau=\infty$ and $K_X(t)=e^{-t}-1+t$. At $\theta=1.5$ the function
$\Psi_{1.5}$ does attain its supremum at the interior point $t_v=7.3296$, but the value
there is $-0.7320$, so $\cstar(1.5)=0$; tangency at $t_v$ would require $\MR(t_v)=1.5$,
whereas $\MR(t_v)=0.2357$. The point $t_v$ therefore does not belong to $\mathcal T$ and
(ii) does not apply to it. By Proposition~\ref{prop:zero} this degeneracy occurs exactly
for spectrally negative laws.
\end{remark}

The set $\mathcal T$ is therefore the whole story, and it is not determined by the control
mechanism. The next two examples are both boundary controlled, both drawn from the
tempered stable family of Example~\ref{ex:ts}, and they fall on opposite sides.

\begin{example}[Exact envelope: no loss at any level]\label{ex:noloss-exact}
Take $Y=-1$, that is $\nu(dx)=Ce^{-Mx}dx$: a centered compound Poisson law with
exponential jumps of mean $1/M$. By Proposition~\ref{prop:ts} the pole is boundary
controlled, $\cstar(V)=1/M=1/\tau$, and $R$ is exactly exponential with mean $1/M$, so
$\MR(t)=(1-t/M)^{-1}$ and
\[
  K_X(t)=\frac{Vt^2}{2}\,\MR(t)=\frac{Vt^2}{2(1-t/M)} .
\]
The envelope is an identity, not a bound. Hence $\Psi_1\equiv1/M=\cstar(V)>0$ on $(0,\tau)$, every
$t\in(0,\tau)$ is a tangency point, $\mathcal T=(0,\tau)$, and
Theorem~\ref{thm:noloss}(iii) gives $y_*(x)=I^{-1}(x)$ on $\{I(K_X'(t)):t\in\mathcal T\}$,
which here is all of $(0,\infty)$: $K_X(t)=Vt^2/\{2(1-t/M)\}$ gives $K_X'(t)\uparrow\infty$
as $t\uparrow M$, so $t\mapsto I(K_X'(t))$ maps $(0,M)$ continuously and increasingly onto
$(0,\infty)$. Equality therefore holds at every level. This is the
equality law of \cite[Theorem 4.3]{Chen2026a} seen from the present angle: boundary
control and losslessness are compatible.
\end{example}

\begin{example}[Genuine loss under boundary control]\label{ex:gaploss}
Take $Y=0$, $C=M=1$, that is a centered exponential variable of unit mean, with $V=1$,
$\tau=1$ and $\cstar\equiv1$ by Example~\ref{ex:gamma}. Here $\mathcal T=\emptyset$.
Indeed $\Psi_\theta\le\Psi_1$ and, expanding,
\[
  \MR(t)=\frac{2\bigl(-\log(1-t)-t\bigr)}{t^2}=\sum_{m\ge0}\frac{2}{m+2}\,t^m
  \;<\;\sum_{m\ge0}t^m=\frac{1}{1-t},\qquad 0<t<1,
\]
because the coefficients satisfy $2/(m+2)\le1$ with equality only at $m=0$. Rearranging,
$\Psi_1(t)<1=\cstar(\theta V)$ for every $t\in(0,1)$ and every $\theta\ge1$, so no
interior maximiser realises the pole. By Theorem~\ref{thm:noloss}(iv) the inequality in
(i) is strict at every level. The frontier being constant, the infimum is attained at
$v=V$ and $y_*(x)=\sqrt{2x}+x$, while $I(y)=y-\log(1+y)$. The relative loss
$1-I^{-1}(x)/y_*(x)$ is
\[
  5.7\%,\ 11.1\%,\ 13.1\%,\ 12.9\%,\ 8.3\%,\ 3.6\%
  \qquad\text{at } x=0.1,\,1,\,5,\,10,\,10^2,\,10^3,
\]
with maximum $13.1\%$ near $x=5.5$.
\end{example}

\begin{remark}
Examples~\ref{ex:noloss-exact} and \ref{ex:gaploss} sit at $Y=-1$ and $Y=0$ of the same
one-parameter family and are both boundary controlled, yet one is exactly lossless and the
other is not. The control mechanism, which decides the shape of the frontier
(Section~\ref{sec:mech}), therefore does not decide whether the optimised envelope attains
the Chernoff bound. That is governed by $\mathcal T$, equivalently by whether the envelope
can be made tangent to $K_X$ at the relevant Legendre maximiser.
\end{remark}

\begin{proposition}[The tangency set is bounded below by the variance-exact maximiser]
\label{prop:Tlower}
Assume the supremum $\sup_{t>0}\Psi_1$ is attained, and set
$t_-=\inf\{t>0:\Psi_1(t)=\sup_{s>0}\Psi_1(s)\}$. Then $\mathcal T\subseteq[t_-,\infty)$.
\end{proposition}

\begin{proof}
Let $t\in\mathcal T$, witnessed by $v=\theta V$ with $c=\cstar(v)$. If $c>0$ then
$\Psi_\theta(t)=c=\sup_{s>0}\Psi_\theta$ by the discussion after \eqref{eq:Tset}; if $c=0$
then tangency reads $\MR(t)=\theta$, so $\Psi_\theta(t)=0$, while $\cstar(v)=0$ forces
$\sup_{s>0}\Psi_\theta\le0$. In both cases $t$ attains $\sup_{s>0}\Psi_\theta$.

Suppose $t<t_-$, and let $t_*$ be any maximiser of $\Psi_1$, so $t<t_-\le t_*$. By
\eqref{eq:pendecomp} and $\Psi_\theta(t)\ge\Psi_\theta(t_*)$,
\[
  \Psi_1(t)\;\ge\;\Psi_1(t_*)+(\theta-1)\{\phi(t)-\phi(t_*)\}\;\ge\;\Psi_1(t_*),
\]
because $\theta\ge1$ and $\phi(t)\ge\phi(t_*)$ by Lemma~\ref{lem:monpen}. Hence $t$ is a maximiser of
$\Psi_1$ with $t<t_-$, contradicting the definition of $t_-$ as the infimum of the
maximisers.
\end{proof}

\begin{remark}\label{rem:tminus}
Attainment of $\sup_{t>0}\Psi_1$ does not make $t_-$ positive, and the infimum defining
$t_-$ need not itself be attained: in Example~\ref{ex:noloss-exact} every point of
$(0,\tau)$ is a maximiser, so $t_-=0$ and Proposition~\ref{prop:Tlower} is vacuous there.
\end{remark}

\begin{corollary}[Maximisers away from the origin force a gap at small deviations]
\label{cor:smallgap}
Assume $\cstar(V)>0$, that $\sup_{t>0}\Psi_1$ is attained, and that its maximisers are
bounded away from the origin, in the sense that the quantity $t_-$ of
Proposition~\ref{prop:Tlower} satisfies $0<t_-<\tau$ (Remark~\ref{rem:tminus} shows that
attainment alone does not give this). Put $x_0=I(K_X'(t_-))$. Then
$y_*(x)>I^{-1}(x)$ for every $x\in(0,x_0)$: no pair on the frontier attains the Chernoff
deviation at any level below $x_0$.
\end{corollary}

\begin{proof}
First, $\Psi_1(t_-)=\cstar(V)$: the function $\Psi_1$ is continuous at $t_-\in(0,\tau)$ by
Lemma~\ref{lem:psicont} and $t_-$ is the infimum of a set of maximisers, so a sequence of
maximisers decreasing to $t_-$ transfers the value. Since $t_-\in(0,\tau)$ we have
$\MR(t_-)<\infty$, so $\cstar(V)=\Psi_1(t_-)<1/t_-$ and
$t_-\in\mathcal T$; Theorem~\ref{thm:noloss}(ii) applied at $t_-$ gives
$I^{-1}(x_0)=K_X'(t_-)$. Let $0<x<x_0$, put $y=I^{-1}(x)$ and $y_0=K_X'(t_-)$. Then $y<y_0$: by
Theorem~\ref{thm:noloss}(ii) $I(y_0)=x_0$, and $I$ is finite and convex on $[0,y_0]$, so
$\limsup_{y'\uparrow y_0}I(y')\le I(y_0)$, while lower semicontinuity gives the reverse
inequality for the $\liminf$; hence $\lim_{y'\uparrow y_0}I(y')=x_0>x$ and some $y'<y_0$ already
satisfies $I(y')>x$, whence $I^{-1}(x)\le y'<y_0$. If $y_*(x)=I^{-1}(x)$, the proof of
Theorem~\ref{thm:noloss}(iv) produces a Legendre maximiser $t_g\in\mathcal T$ of $K_X$ at
$y$; then $K_X'(t_g)=y<K_X'(t_-)$ and $K_X'$ is strictly increasing on $(0,\tau)$, so
$t_g<t_-$, contradicting Proposition~\ref{prop:Tlower}. Equality therefore fails, and
Theorem~\ref{thm:noloss}(i) gives the strict inequality.
\end{proof}

Example~\ref{ex:twopoint}, taken with intensity $\lambda=1$ so that $V=1.6$, is interior
controlled with $t_-=0.4737$ and
$x_0=I(K_X'(t_-))=0.5415$. Numerically the relative excess $y_*(x)/I^{-1}(x)-1$ is
$0.149\%$ at $x=0.01$, $0.180\%$ at $x=0.05$ and $0.083\%$ at $x=0.2$, and vanishes to
machine accuracy at $x=x_0$, as Theorem~\ref{thm:noloss}(ii) applied at $t_-$ predicts, and
also at every larger level tested, which the theory does not require.
The gap is small but genuine, and it is invisible if one computes only at moderate levels.

\begin{corollary}[Strict pure local control makes the variance-exact point suboptimal]
\label{cor:neveropt}
Assume the hypotheses of Theorem~\ref{thm:edge} or those of Theorem~\ref{thm:cube}; more
generally, assume only that $\cstar'(V^+)=-\infty$. Then
\[
  y_*(x)<\sqrt{2Vx}+\cstar(V)\,x\qquad\text{for every }x>0 .
\]
Under boundary control, by contrast, the infimum in \eqref{eq:ystar} is attained at $v=V$
for every $x>0$.
\end{corollary}

\begin{proof}
Write $h(v)=\sqrt{2vx}+\cstar(v)x$ and $\rho(\varepsilon)=\cstar(V)-\cstar(V+\varepsilon V)$,
so that
\[
  h(V+\varepsilon V)-h(V)=\sqrt{2Vx}\bigl(\sqrt{1+\varepsilon}-1\bigr)-x\rho(\varepsilon)
  =O(\varepsilon)-x\rho(\varepsilon).
\]
The hypothesis $\cstar'(V^+)=-\infty$ says exactly that
$\rho(\varepsilon)/\varepsilon\to\infty$, so the right-hand side is negative for all small
$\varepsilon>0$; hence $h$ takes values below $h(V)$ and the infimum is strictly smaller.
Theorem~\ref{thm:edge} gives $\rho(\varepsilon)\sim2\sqrt{\kappa\varepsilon}$ and
Theorem~\ref{thm:cube} gives $\rho(\varepsilon)\sim3(\gamma\varepsilon^2/4)^{1/3}$, both of
order larger than $\varepsilon$. Under boundary control, by contrast, $\cstar$ is constant
by Proposition~\ref{prop:floor}, so $h$ is strictly increasing and its infimum is attained
at $v=V$.
\end{proof}

\begin{remark}
The hypothesis $m_1>1/\tau$ is exactly what separates the two halves of the corollary: it
gives $\cstar(V)=m_1>1/\tau$, whereas boundary control means $\cstar(V)=1/\tau$. It is a
genuine assumption, not a consequence of pure local control, which forces only
$m_1\ge1/\tau$ (Remark~\ref{rem:implied}), and it cannot be omitted:
Example~\ref{ex:knifefloor} is purely locally controlled with $m_1=1/\tau$, and falls under
the second half. What drives the first half is only the infinite one-sided derivative, and
that hypothesis is essentially necessary: since $h'(V^+)=\sqrt{x/(2V)}+\cstar'(V^+)x$, a
finite slope would leave $v=V$ locally optimal at all small enough $x$. Interior control
does produce a finite slope --- near an interior maximiser $t_-$ one has
$\sup_t\Psi_\theta\approx\Psi_1(t_-)-\varepsilon\phi(t_-)$ by \eqref{eq:pendecomp} --- which
is consistent with Corollary~\ref{cor:smallgap} and with the numerics for
Example~\ref{ex:twopoint}, where the optimal proxy stays at $v=V$ at small levels.
\end{remark}

Table~\ref{tab:noloss} illustrates both statements for the law of Example~\ref{ex:g4},
normalised so that $V=1$. The column ``naive'' is the deviation obtained from the
variance-exact point alone; ``best'' is the minimum in \eqref{eq:ystar}; ``Chernoff'' is
$I^{-1}(x)$ computed directly from $K_X$. The optimal proxy is reported to two decimals
only: the objective in \eqref{eq:ystar} is very flat near its minimiser, and at $x=100$ a
change of $0.01$ in $v$ moves the minimum by less than $10^{-4}$.

\begin{table}[t]
\centering
\caption{Deviation reaching confidence $e^{-x}$ for the centered compound Poisson law with
$\mathrm{Gamma}(4,1)$ jumps, normalised to $V=1$. The optimised sub-gamma bound coincides
numerically with the exact Chernoff deviation in every row.}
\label{tab:noloss}
\begin{tabular}{rrrrrr}
\toprule
$x$ & naive $(v=V)$ & optimal $v$ & best & Chernoff & saving \\
\midrule
$0.25$ & $1.2071$   & $1.09$ & $1.1557$  & $1.1557$  & $4.3\%$ \\
$0.5$  & $2.0000$   & $1.15$ & $1.8738$  & $1.8738$  & $6.3\%$ \\
$1$    & $3.4142$   & $1.23$ & $3.1108$  & $3.1108$  & $8.9\%$ \\
$2$    & $6.0000$   & $1.34$ & $5.2861$  & $5.2861$  & $11.9\%$ \\
$5$    & $13.1623$  & $1.56$ & $11.0159$ & $11.0159$ & $16.3\%$ \\
$10$   & $24.4721$  & $1.81$ & $19.6382$ & $19.6382$ & $19.8\%$ \\
$25$   & $57.0711$  & $2.28$ & $43.2717$ & $43.2717$ & $24.2\%$ \\
$100$  & $214.1421$ & $3.54$ & $149.5782$& $149.5782$& $30.2\%$ \\
\bottomrule
\end{tabular}
\end{table}

The pattern is what Theorem~\ref{thm:noloss} predicts; that the eight levels lie in
$I(K_X'(\mathcal T))$ is checked numerically, not derived. As $x$ grows the optimal
proxy moves to the right along the frontier, tracking the Legendre maximiser as it
approaches $\tau$; the saving relative to the variance-exact point grows from four percent
in the near-Gaussian regime to thirty percent at $x=100$. In the opposite direction, as
$x\downarrow0$ the optimal proxy returns to $V$, which is the sense in which the
variance-exact normalisation is correct: it is optimal only in the limit of small
deviations, and by Corollary~\ref{cor:neveropt} it is not optimal at any positive
deviation level for this law.

\begin{remark}
Theorem~\ref{thm:noloss} qualifies a common reading of sub-gamma bounds. The gap between a
Bernstein bound and the exact Chernoff bound is often attributed to the crudeness of the
quadratic-over-linear envelope. For the law of Example~\ref{ex:g4} the
envelope family is not crude at all: it is exactly sharp at every level tested once both
of its parameters are allowed to move along the feasible frontier, and what is crude is
the habit of reporting a single pair. We do not claim this for every such law, since it
depends on $\mathcal T$, which we do not determine in general: Example~\ref{ex:gaploss}
exhibits a law for which no pair on the frontier reaches the Chernoff bound at any level,
and Corollary~\ref{cor:smallgap} shows that a law whose variance-exact maximisers are
bounded away from the origin always fails below an explicit threshold. In those cases the residual gap is a limitation of the
quadratic-over-linear shape rather than of the normalisation. Theorem~\ref{thm:noloss}(iv) says exactly which case one is in.
\end{remark}

\section{Discussion}

Fixing the quadratic proxy at the true variance makes the sub-gamma pole a well-defined
functional of the law, and the preceding papers of this series computed it. The present
paper shows that this normalisation is a boundary point of a one-parameter problem, and
that the boundary point is singular. Three consequences seem worth emphasising.

First, the two endpoint obstructions are not symmetric. A boundary-controlled law has a
completely flat frontier, so for such laws the variance-exact answer is the whole answer.
A purely locally controlled law whose local obstruction strictly dominates the floor has a
frontier that leaves its endpoint with infinite slope --- a square root when $\kappa>0$, a
cube root in the remaining case $\kappa=0$, $\gamma>0$, and an
$\varepsilon^{j/(j+1)}$ law in general --- so for such laws the variance-exact answer is
the least informative point of the curve. The
third-cumulant obstruction, which is the most quoted feature of the variance-exact
problem, exists only at $v=V$.

Second, the phase transitions found in the variance-exact problem are transitions of the
same knife-edge type. The Skellam threshold $p_+=(2+\sqrt3)/4$ separates local from
interior control of the right pole at $v=V$, with $p_-=(2-\sqrt3)/4$ playing the same role
for the left pole by symmetry, and neither has an analogue for $v>V$. What the transition
does leave behind is a change in the edge exponent: the frontier leaves $v=V$ like
$\sqrt{\theta-1}$ for $p>p_+$ and like $(\theta-1)^{2/3}$ at $p=p_+$ itself
(Theorems~\ref{thm:edge} and \ref{thm:cube}). This suggests that such
thresholds should be read as statements about a normalisation rather than about the law.

Third, losslessness of the two-parameter family is a separate question from the shape of
the frontier, and Theorem~\ref{thm:noloss}(iv) reduces it to a single set: the collection
$\mathcal T$ of interior points at which some envelope on the frontier is tangent to
$K_X$. Where the relevant Legendre maximiser lies in $\mathcal T$, optimising over the
frontier reproduces the Chernoff bound exactly and nothing is given up by working with
sub-gamma envelopes, provided one works with all of them; where it does not, the gap is
real. Examples~\ref{ex:noloss-exact} and \ref{ex:gaploss} show that the control mechanism
does not decide between the two globally, since two boundary-controlled members of one
family fall on opposite sides. In practice this means reporting the frontier, or at least $\cstar(V)$ together with the
interval on which the curve is non-constant --- when $\tau<\infty$ that interval is
$[V,v^*)$ with $v^*$ as in Proposition~\ref{prop:vstar} --- rather than a single pole.

Several questions remain. The hypotheses that drive Section~\ref{sec:mech} and
Section~\ref{sec:asym} --- $\cstar(V)>\max\{0,1/\tau\}$ in Theorem~\ref{thm:drop}, and pure
local control \eqref{eq:purelocal} in Theorem~\ref{thm:edge} --- are conditions on the
function $\Psi_1$ rather than on the L\'evy triplet, and we verify them family by family
(Proposition~\ref{prop:ts}); a triplet-level criterion for either would be worth having.
The critical proxy $v^*$ of Proposition~\ref{prop:vstar} is
finite when $\MR(\tau^-)<\infty$, but we have no description of $v^*$ in terms of the
L\'evy triplet either. We do not determine $\mathcal T$ in general. It is all of $(0,\tau)$ when
the envelope is an identity (Example~\ref{ex:noloss-exact}), empty for the centered
exponential (Example~\ref{ex:gaploss}), empty again for Example~\ref{ex:knifefloor}, and
appears to be $(0,\tau)$ in Example~\ref{ex:g4}; it is bounded away from the origin whenever
the maximisers at $v=V$ are
(Proposition~\ref{prop:Tlower}), so a gap is forced at every small enough level
(Corollary~\ref{cor:smallgap}); a description in terms of the L\'evy triplet, together
with the levels $I(K_X'(\mathcal T))$ it covers, is the natural next step and would settle
exactly when the sub-gamma family is lossless. Finally, the multivariate frontier of
Section~\ref{sec:multi} is a family of positively homogeneous functions indexed by
$\theta$, and the geometric questions studied in \cite{ChenWang2026b} for $\theta=1$, in
particular subadditivity, may behave better for $\theta>1$, where local control is
unavailable.

\section*{Declaration of competing interest}

The authors declare that they have no known competing financial interests or personal
relationships that could have appeared to influence the work reported in this paper.

\section*{CRediT authorship contribution statement}

\textbf{Yichuan Chen:} Conceptualization, Methodology, Formal analysis, Investigation,
Software, Writing -- original draft. \textbf{Xin Wang:} Supervision, Writing -- review \&
editing.

\section*{Funding}

This work received no specific grant from any funding agency in the public, commercial,
or not-for-profit sectors.

\section*{Data availability}

This article has no associated data: no datasets were generated or analysed. Every
numerical value reported in Section~\ref{sec:examples} and Table~\ref{tab:noloss} is
computed from the closed-form expressions stated in the text.

\section*{Declaration of generative AI and AI-assisted technologies in the manuscript
preparation process}

During the preparation of this work the authors used Claude (Anthropic) for: language
editing; adversarial checking of the proofs and of the numerical computations, a process
that identified errors in earlier drafts and led to revisions of several results; and a
partial Lean~4 formalization of the results of Sections~\ref{sec:setting}--\ref{sec:frontier}.
The tool was not used to generate mathematical results. The authors reviewed and edited
all output of the tool, verified every mathematical statement independently, and take
full responsibility for the content of this publication.

\end{document}